\documentclass[11pt]{amsart}

\usepackage[T1]{fontenc}
\usepackage{lmodern}
\usepackage[expansion=false]{microtype}
\usepackage{amsmath,amssymb,amsthm,mathtools}
\usepackage{mathrsfs}
\usepackage{booktabs}
\usepackage{tikz}
\usetikzlibrary{arrows.meta}
\usepackage[hidelinks]{hyperref}
\hypersetup{
  pdftitle={Joint level--weight murmurations: prime averaging and the cubic pointwise range},
  pdfauthor={Julien Cardi},
  pdfsubject={Analytic number theory and modular forms},
  pdfkeywords={holomorphic newforms, murmurations, Atkin--Lehner trace,
    Hurwitz class numbers, character sums}
}

\numberwithin{equation}{section}
\newtheorem{theorem}{Theorem}[section]
\newtheorem{proposition}[theorem]{Proposition}
\newtheorem{lemma}[theorem]{Lemma}
\newtheorem{corollary}[theorem]{Corollary}
\theoremstyle{definition}

\theoremstyle{remark}
\newtheorem{remark}[theorem]{Remark}

\newcommand{\R}{\mathbb R}
\newcommand{\eps}{\varepsilon}
\newcommand{\1}{\mathbf 1}
\newcommand{\M}{\mathcal M}
\newcommand{\Kern}{\mathcal K}
\providecommand{\doi}[1]{doi: \href{https://doi.org/#1}{#1}}

\title[Joint level--weight murmurations]
{Joint level--weight murmurations: prime averaging and the cubic pointwise range}
\author{Julien Cardi}
\address{Independent researcher, Marseille, France}
\email{265julien@gmail.com}
\date{}

\begin{document}

\begin{abstract}
Let $N$ range over squarefree levels $N\asymp X$, and let the even weight $k$
vary in a smooth window $k\asymp K$.  We study natural root-number-weighted
traces of Hecke eigenvalues at primes $p\asymp XK^2$.  First, uniformly for
$X^{\eps_0}\leq K\leq X^{3-\eps_0}$, we prove a fixed-prime asymptotic whose
main term is expressed through Zubrilina's murmuration density.  The exponent
$3$ is the endpoint of our absolute treatment of the nonzero Poisson
frequencies.  Second, after averaging the primes with logarithmic weight, we
obtain unconditionally, throughout every fixed polynomial range
$X^{\eps_0}\leq K\leq X^{A_0}$, an explicit atomic limiting measure.  The
fixed-prime argument groups the local Fourier expansion by exact additive
conductor and uses Burgess truncation, whereas the prime-averaged argument
applies a smooth Barban--Davenport--Halberstam estimate.
\end{abstract}

\maketitle

\section{Introduction}

For a normalized holomorphic newform $f$ of squarefree level $N$, even
weight $k$, and trivial character, write $\lambda_f(n)$ for its normalized
Hecke eigenvalues and $\epsilon_f\in\{\pm1\}$ for its root number.
Murmurations are correlations between $\lambda_f(p)$ and $\epsilon_f$ when
the prime $p$ is placed at the scale of the analytic conductor.  The
phenomenon was first observed experimentally for elliptic curves by He,
Lee, Oliver, and Pozdnyakov \cite{HLOPMurmurations}.  Zubrilina
\cite{ZubrilinaMurmurations} proved a level-aspect correlation for
holomorphic newforms of fixed weight.

The complementary weight aspect at level one was developed by Bober,
Booker, Lee, and Lowry-Duda \cite[Theorem~1.1]{BBLLMurmurations}, who use
natural counting weights and a sharp weight interval; their stated
prime-averaged theorem assumes GRH for Dirichlet and modular $L$-functions.
Kuan and Lesesvre \cite[Theorem~1]{KuanLesesvre} use Petersson weights and a
smooth Gaussian window, and assume GRH only for Dirichlet $L$-functions.
The two spectral weights lead to different limiting densities.  Related
directions now include Maass forms \cite{BookerEtAlMaass}, $p$-power
coefficients \cite{KunduMuller}, ratios-conjecture and
approximate-functional-equation methods \cite{CowanRatios}, the depth
aspect \cite{BurrinDepth,TomczakDepth}, and elliptic curves ordered by
height \cite{SawinSutherland}.

Here the squarefree level $N\asymp X$ and the weight $k\asymp K$ tend to
infinity together, with natural counting weights.  The contribution of the
paper is fourfold.  First, the level average gives a fixed-prime asymptotic,
uniform up to the cubic range $K\leq X^{3-\eps_0}$.  Second, a subsequent
logarithmically weighted prime average yields an unconditional atomic law
throughout every fixed polynomial range in $K$.  Third, the trace formula is
used directly, before introducing harmonic weights, so the resulting density
belongs to the natural family rather than to a Petersson-weighted family.
Finally, the two ranges are traced to two different analytic mechanisms.  At
a fixed prime, Poisson summation in the level produces nonzero additive
frequencies.  When the prime is averaged first, these frequencies are replaced
by the mean-square distribution of primes in arithmetic progressions.

Fix two nonnegative, nonzero smooth functions $W$ and $V$, compactly
supported in $(0,\infty)$.  For a smooth compactly supported function
$F$ on $(0,\infty)$, put
\begin{align}
 \mathcal N_{X,K}(F)
  ={}&\sum_{N\geq1}^{\mathrm{sf}}W(N/X)
  \sum_{\substack{k\geq2\\ k\ {\rm even}}}V(k/K)
  \sum_{p\nmid N}(\log p)\sqrt p
  \sum_{f\in H_k^*(N)}\lambda_f(p)\epsilon_f \notag\\
  &\hspace{35mm}\times
  F\!\left(\frac{p}{N((k-1)/(4\pi))^2}\right),
       \label{eq:intro-numerator}\\
 \mathcal C_{X,K}
  ={}&\sum_{N\geq1}^{\mathrm{sf}}W(N/X)
  \sum_{\substack{k\geq2\\ k\ {\rm even}}}V(k/K) \notag\\
  &\hspace{35mm}\times
  N\left(\frac{k-1}{4\pi}\right)^2
  \dim S_k^{\rm new}(N).                         \label{eq:intro-normalizer}
\end{align}
Here and below a superscript $\mathrm{sf}$ restricts a sum to squarefree
integers, and $p$ always denotes a prime.  The normalizer has order
$X^3K^4$.

For a prime $\ell$, set
\[
 D_\ell=\ell^4-2\ell^2-\ell+1,
 \qquad
 Q(d)=\mu^2(d)\prod_{\ell\mid d}\frac{\ell^2}{D_\ell},
\]
and define
\begin{equation}\label{eq:W0-intro}
 W_0=\prod_\ell
 \frac{(\ell^2-1)^2}{\ell^4-2\ell^2+\ell}.
\end{equation}
The product in \eqref{eq:W0-intro} is convergent.  The measure appearing in
the limit is
\begin{equation}\label{eq:atomic-measure-intro}
 \mu_{\rm at}
 =\sum_{\substack{q\geq1\ {\rm squarefree},\ a\geq1\\(q,a)=1}}
 W_0\left(\frac qa\right)^4
 \prod_{\ell\mid q}\frac{\ell}{(\ell^2-1)^2}
 \delta_{(q/a)^2}.
\end{equation}
It is locally finite on $(0,\infty)$.  Indeed,
\[
 \prod_{\ell\mid q}\frac{\ell}{(\ell^2-1)^2}\ll q^{-3},
\]
and on a fixed compact set of values of $(q/a)^2$ there are $O(q)$
possible integers $a$ for each $q$.  The resulting majorant is
$\sum_q q^{-2}$.

\begin{theorem}[Prime-averaged murmuration law]\label{thm:main}
Fix $\eps_0,A_0>0$, $F\in C_c^\infty((0,\infty))$, and $B>0$.
Uniformly for $X^{\eps_0}\leq K\leq X^{A_0}$,
\begin{equation}\label{eq:main-theorem}
 \frac{\mathcal N_{X,K}(F)}{\mathcal C_{X,K}}
 =\int_0^\infty F(v)\,d\mu_{\rm at}(v)
 +O_{B,F,V,W,\eps_0,A_0}\big((\log X)^{-B}\big).
\end{equation}
Moreover,
\begin{equation}\label{eq:denominator-theorem}
\begin{split}
 \frac{1}{\mathcal C_{X,K}}
 \sum_{N\geq1}^{\mathrm{sf}}W(N/X)
 \sum_{\substack{k\geq2\\k\ {\rm even}}}V(k/K)
 \dim S_k^{\rm new}(N)
 \sum_{p\nmid N}(\log p)\\
 {}\times
 F\!\left(\frac{p}{N((k-1)/(4\pi))^2}\right)
 =\int_0^\infty F(v)\,dv+O_{B,F,V,W,\eps_0,A_0}((\log X)^{-B}).
\end{split}
\end{equation}
\end{theorem}

For a prime $p$, let $\mathcal T_{X,K}(p;F)$ denote the summand of
\eqref{eq:intro-numerator} before the prime sum and its factor $\log p$:
\begin{align}
 \mathcal T_{X,K}(p;F)
 ={}&\sum_{N\geq1}^{\mathrm{sf}}W(N/X)
 \sum_{\substack{k\geq2\\ k\ {\rm even}}}V(k/K)
 F\!\left(\frac{16\pi^2p}{N(k-1)^2}\right) \notag\\
 &\hspace{24mm}\times
 \sqrt p\sum_{f\in H_k^*(N)}\lambda_f(p)\epsilon_f.
                                                        \label{eq:pointwise-trace}
\end{align}
The main term is expressed through Zubrilina's density
\begin{equation}\label{eq:Mk-intro}
 \M_k(y)=\alpha\sqrt y\sum_{d,s\geq1}
          \frac{Q(d)}sJ_{k-1}\!\left(\frac{4\pi s\sqrt y}{d}\right),
\end{equation}
where $J_\nu$ is the Bessel function and
\begin{equation}\label{eq:alpha-intro}
 \alpha=2\pi\prod_\ell
 \frac{\ell^4-2\ell^2-\ell+1}{\ell^4-2\ell^2+\ell}.
\end{equation}

\begin{theorem}[Pointwise trace in the cubic range]
\label{thm:pointwise-main}
Fix $0<\eps_0<3/2$ and the functions $F,V,W$ above.  There is
$\delta=\delta(\eps_0)>0$ such that, uniformly for primes $p\asymp XK^2$ and
\[
                 X^{\eps_0}\leq K\leq X^{3-\eps_0},
\]
one has
\begin{align}
 \mathcal T_{X,K}(p;F)
 ={}&\frac{\mathfrak c_\varphi}{12}
 \sum_{\substack{k\geq2\\k\ {\rm even}}}(k-1)V(k/K)
 \notag\\
 &\quad\times
 \int_0^\infty tW(t/X)
 F\!\left(\frac{16\pi^2p}{t(k-1)^2}\right)
 \M_k(p/t)\,dt \notag\\
 &+O_{F,V,W,\eps_0}(X^{2-\delta}K^2).
                                                        \label{eq:pointwise-main}
\end{align}
\end{theorem}

To compare the two statements, write $K=X^\rho$.  The point $\rho=3$ in
the fixed-prime range is the endpoint of the present absolute-value
estimate, not a claim that a pointwise asymptotic beyond it is impossible.

We use the trace formula before introducing harmonic weights.  For
$p\nmid N$, the spectral sum in \eqref{eq:intro-numerator} is the trace of
$T_pW_N$.  Its elliptic part contains
\[
 (-1)^{k/2-1}U_{k-2}\!\left(\frac{r\sqrt N}{2\sqrt p}\right)
 H_1(N^2r^2-4Np).
\]
If $p\asymp XK^2$, summation over $k$ makes this expression rapidly
decreasing in $r$.  The relevant discriminants are consequently of size
$X^2K^2$, uniformly away from zero.

The exact trace formula reduces both theorems to averages of Hurwitz class
numbers.  Burgess's estimate permits the Dirichlet series for the quadratic
$L$-values to be cut at
\[
                   T=(XK)^{1/2+\eta}.
\]
At a fixed prime, Poisson summation in the level gives a relative error
\[
 X^\sigma\left(Z^{-1}+\frac{ZT^{1/2}}X\right).
\]
This is a power saving below the cubic point.  In the prime-averaged
problem, the characters have modulus $8f^2m$ as functions of $p$.  The
smooth Barban--Davenport--Halberstam theorem gives instead
\[
 D\left(\frac{T}{XK^2}\right)^{1/2}(\log X)^C,
\]
which decreases as $K$ grows.  These estimates account for the two ranges
in Figure~\ref{fig:ranges}.

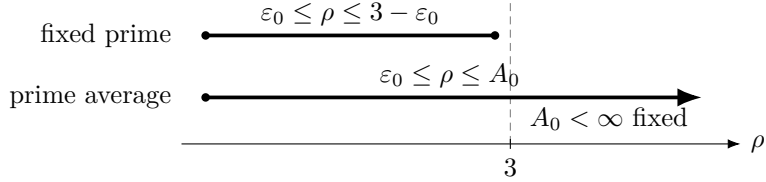
\begin{figure}[t]
\centering
\begin{tikzpicture}[x=1.45cm,y=.82cm,>=Latex,font=\small]
  \draw[->] (0,0) -- (5.1,0) node[right] {$\rho$};
  \draw[densely dashed,gray] (3,0) -- (3,2.35);
  \draw (3,.08) -- (3,-.08) node[below] {$3$};
  \node[left] at (0,1.75) {fixed prime};
  \node[left] at (0,.75) {prime average};
  \draw[line width=1.4pt] (.22,1.75) -- (2.86,1.75);
  \fill (.22,1.75) circle (1.6pt);
  \fill (2.86,1.75) circle (1.6pt);
  \node[above] at (1.54,1.75) {$\eps_0\leq\rho\leq3-\eps_0$};
  \draw[->,line width=1.4pt] (.22,.75) -- (4.75,.75);
  \fill (.22,.75) circle (1.6pt);
  \node[above] at (2.45,.75) {$\eps_0\leq\rho\leq A_0$};
  \node[below] at (3.9,.75) {$A_0<\infty$ fixed};
\end{tikzpicture}
\caption{Ranges of the weight parameter.  Here $A_0$ is arbitrary but
fixed before $X\to\infty$.}
\label{fig:ranges}
\end{figure}

For the identity used below, set
\begin{equation}\label{eq:beta-intro}
 \beta=2\pi\prod_\ell
 \frac{\ell^3+\ell^2-1}{\ell(\ell^2+\ell-1)},
\end{equation}
and
\begin{equation}\label{eq:nu-gamma-intro}
 \nu(r)=\prod_{\ell\mid r}
 \left(1+\frac{\ell^2}{D_\ell}\right),
 \qquad
 \gamma=12\prod_\ell\frac{\ell(\ell+1)}{\ell^2+\ell-1}.
\end{equation}

\begin{proposition}[Zubrilina's identity]\label{prop:zubrilina-identity}
For every even $k\geq2$ and $y>0$, the function in
\eqref{eq:Mk-intro} satisfies
\begin{align}
 \M_k(y)={}&
 \frac{\alpha(-1)^{k/2-1}}{k-1}
 \sum_{1\leq r<2\sqrt y}\nu(r)\sqrt{4y-r^2}\,
 U_{k-2}\!\left(\frac r{2\sqrt y}\right)\notag\\
 &+\frac{\beta}{k-1}\sqrt y-\gamma\delta_{k,2}y.
                                                        \label{eq:Mk-two-forms}
\end{align}
\end{proposition}

\begin{proof}
This is \cite[Theorem~3]{ZubrilinaMurmurations}, combined with the
definition of the fixed-weight density in
\cite[Theorem~1]{ZubrilinaMurmurations}.  The normalizations of
$Q,\alpha,\beta,\gamma$, and $\nu$ are those in
\eqref{eq:Mk-intro}--\eqref{eq:nu-gamma-intro}.  The endpoint in the
$r$-sum contributes zero and may equivalently be included.
\end{proof}
The sum over the weight is now untwisted.  A smoothed form of Neumann's
identity
\[
 \sum_{n\geq1,\ n\ {\rm odd}}nJ_n(x)=\frac x2
\]
replaces the Bessel series by point masses.  Grouping $d$ and $s$ by their
greatest common divisor gives exactly \eqref{eq:atomic-measure-intro}.

Section~\ref{sec:notation} fixes the normalizations, and
Section~\ref{sec:exact-trace} gives the trace formula.  The uniform kernel
and the class-number truncation are treated in
Sections~\ref{sec:weight-kernel} and~\ref{sec:class-number}.  Three arithmetic
multiplicity statements are isolated explicitly: stabilization of the local
coefficients, the $O(m)$ support bound in the prime variable, and the exact
lifting identity for the fixed-prime zero frequency.  We prove the
prime-averaged estimate in Section~\ref{sec:prime-first} and the pointwise
estimate in Sections~\ref{sec:exact-conductor}
and~\ref{sec:pointwise-proof}.  The final section identifies the
prime-averaged main term with the atomic measure
\eqref{eq:atomic-measure-intro}.

\section{Notation and elementary estimates}\label{sec:notation}

We write $e(x)=\exp(2\pi i x)$, and $\Lambda$ denotes the von Mangoldt
function.  Throughout, $\eps>0$ is an arbitrarily small fixed number.
Implied constants may depend on the fixed functions $W,V,F$; dependence on
$\eps$ and on the auxiliary parameters $\eta,\vartheta,\xi$ is displayed
when it matters.  An integer $A$ used in a rapid-decay estimate is indicated
by a subscript.
We write $n=k-1$; hence $n$ is odd.  After enlarging fixed compact
intervals, the support conditions in \eqref{eq:intro-numerator} imply
\begin{equation}\label{eq:basic-scales}
 N\asymp X,\qquad n\asymp K,\qquad p\asymp XK^2.
\end{equation}
We use throughout the two natural composite scales
\begin{equation}\label{eq:P-Y-scales}
                         P=XK^2,\qquad Y=XK.
\end{equation}
The auxiliary parameters are kept in the following ledger; their final
numerical choices are made only when the two arguments separate.
\begin{center}
\small
\begin{tabular}{@{}lll@{}}
\toprule
parameter & size & role \\
\midrule
$P$ & $XK^2$ & prime scale \\
$Y$ & $XK$ & square root of the discriminant scale \\
$T$ & $Y^{1/2+\eta}$ & Burgess truncation \\
$D$ & $X^\vartheta$ & square-divisor truncation \\
$R$ & $X^\vartheta$ & elliptic-index truncation \\
$Z$ & chosen in Section~\ref{sec:pointwise-proof} & squarefree-sieve cutoff \\
\bottomrule
\end{tabular}
\end{center}
The prime-averaged argument requires $8D^2T=o(P)$, whereas the fixed-prime
argument balances the sieve tail $Z^{-1}$ against the nonzero-frequency term
$ZT^{1/2}/X$.  The displayed choices later leave a fixed power of $X$ in
both estimates.

Since $K\geq X^{\eps_0}$, the prime $p$ is larger than $N$ for all sufficiently
large $X$.  Thus $p\nmid N$ automatically in every sum to which the trace
formula is applied.

Put
\begin{equation}\label{eq:cphi}
 \mathfrak c_\varphi
 =\frac1{\zeta(2)}\prod_\ell\left(1-\frac1{\ell^2+\ell}\right).
\end{equation}

\begin{lemma}\label{lem:phi-average}
If $G\in C_c^\infty((0,\infty))$, then
\begin{align}
 \sum_{m\geq1}^{\mathrm{sf}}\varphi(m)G(m/X)
  &=\mathfrak c_\varphi X^2\int_0^\infty uG(u)\,du
    +O_G(X^{3/2+\eps}),                                      \label{eq:phi-average}\\
 \sum_{m\geq1}^{\mathrm{sf}}m\varphi(m)G(m/X)
  &=\mathfrak c_\varphi X^3\int_0^\infty u^2G(u)\,du
    +O_G(X^{5/2+\eps}).                                      \label{eq:nphi-average}
\end{align}
\end{lemma}

\begin{proof}
The Dirichlet series of $\mu^2(m)\varphi(m)$ is
\[
 \prod_\ell\left(1+\frac{\ell-1}{\ell^s}\right)
 =\zeta(s-1)\prod_\ell
  \left(1-\ell^{1-s}\right)
  \left(1+(\ell-1)\ell^{-s}\right).
\]
The second product is absolutely convergent for $\Re s>3/2$ and has value
$\prod_\ell(1-(\ell^2+\ell)^{-1})/\zeta(2)$ at $s=2$ after the residue of
$\zeta(s-1)$ is taken.  Perron's formula with a smooth weight, shifted to
$\Re s=3/2+\eps$, proves \eqref{eq:phi-average}.  Partial summation gives
\eqref{eq:nphi-average}.
\end{proof}

For later use, define
\begin{equation}\label{eq:VK}
 I_V(K)=K\int_0^\infty (Kt-1)^3V(t)\,dt.
\end{equation}

\begin{lemma}\label{lem:weight-cubic-sum}
For every $A>0$,
\begin{equation}\label{eq:weight-cubic-sum}
 \sum_{\substack{n\geq1\\n\ {\rm odd}}}n^3V((n+1)/K)
 =\frac12 I_V(K)+O_{A,V}(K^{-A}).
\end{equation}
In particular, when $V$ is nonnegative and nonzero,
$I_V(K)\asymp K^4$ for all sufficiently large $K$.
\end{lemma}

\begin{proof}
Apply Poisson summation on the progression $n\equiv1\pmod2$ to the smooth
compactly supported function $t^3V((t+1)/K)$.  All nonzero frequencies are
$O_{A,V}(K^{-A})$ after increasing $A$ to absorb the polynomial factor.
\end{proof}

\begin{proposition}[The normalizer]\label{prop:normalizer}
Fix $\eps_0,A_0>0$.  Uniformly for
$X^{\eps_0}\leq K\leq X^{A_0}$, and for every $\eps>0$, one has
\begin{equation}\label{eq:normalizer-asymptotic}
 \mathcal C_{X,K}
 =\frac{\mathfrak c_\varphi}{384\pi^2}
 X^3 I_V(K)\int_0^\infty u^2W(u)\,du
 +O_{W,V,\eps}\left(X^{5/2+\eps}K^4+X^{2+\eps}K^3\right).
\end{equation}
\end{proposition}

\begin{proof}
For squarefree $N$, the dimension formula
\cite{MartinDimension} gives
\begin{equation}\label{eq:dimension-formula}
 \dim S_k^{\rm new}(N)=\frac{k-1}{12}\varphi(N)+O_\eps(N^\eps).
\end{equation}
Insert \eqref{eq:dimension-formula} in \eqref{eq:intro-normalizer}, and use
Lemmas~\ref{lem:phi-average} and \ref{lem:weight-cubic-sum}.  The error in
\eqref{eq:dimension-formula} contributes $O_\eps(X^{2+\eps}K^3)$, and the
error in \eqref{eq:nphi-average} contributes $O_\eps(X^{5/2+\eps}K^4)$.
Relative to the main term, these errors are
$O(X^{-1+\eps}K^{-1})$ and $O(X^{-1/2+\eps})$.  Taking $\eps>0$
sufficiently small in terms of $\eps_0$ proves the proposition.  No upper
bound on $K$ is used.
\end{proof}

\section{The Atkin--Lehner trace}\label{sec:exact-trace}

We fix the normalizations used below.  If
\(\alpha=\left(\begin{smallmatrix}a&b\\c&d\end{smallmatrix}\right)\in
\mathrm{GL}_2^+(\mathbb Q)\), put
\[
 (f[\alpha]_k)(z)=\det(\alpha)^{k-1}(cz+d)^{-k}f(\alpha z).
\]
The Hecke operator \(T_n\) is the double-coset operator associated with
matrices of determinant \(n\) which are upper triangular modulo \(N\).
Thus, if
\[
 f(z)=\sum_{n\geq 1}a_f(n)e(nz),\qquad a_f(1)=1,
\]
is a Hecke eigenform, then \(T_n f=a_f(n)f\).  The normalized Fricke
involution is
\[
 (W_Nf)(z)=N^{-k/2}z^{-k}f\left(-\frac1{Nz}\right).
\]
These are the conventions of \cite[Section~2]{AssafTrace}.

Write
\[
 \lambda_f(n)=a_f(n)n^{-(k-1)/2},\qquad
 L(s,f)=\sum_{n\geq1}\frac{\lambda_f(n)}{n^s}.
\]
With
\[
 \Lambda(s,f)=N^{s/2}(2\pi)^{-s}
 \Gamma\left(s+\frac{k-1}{2}\right)L(s,f),
\]
we write \(\Lambda(s,f)=\epsilon_f\Lambda(1-s,f)\).  If
\(W_Nf=w_N(f)f\), Mellin inversion gives
\begin{equation}\label{eq:root-fricke}
 \epsilon_f=i^k w_N(f)=(-1)^{k/2}w_N(f).
\end{equation}
For squarefree \(N\), this is equivalent to
\[
 w_N(f)=(-1)^{\omega(N)}\sqrt N\,\lambda_f(N).
\]

\begin{remark}[Specialization dictionary]\label{rem:trace-dictionary}
The formulas below use the normalization just fixed.  In the specialization
of Popa's trace formula we take
\[
                 Q=N,\qquad n=p,\qquad w=k-2,\qquad t=Nr.
\]
The polynomial $p_w(t,n)$ is defined by
$\sum_{w\geq0}p_w(t,n)z^w=(1-tz+nz^2)^{-1}$, and hence
\[
 \frac{p_{k-2}(Nr,Np)}{N^{(k-2)/2}}
 =p^{(k-2)/2}U_{k-2}\!\left(\frac{r\sqrt N}{2\sqrt p}\right).
\]
Together with $T_pf=a_f(p)f$ and
$\epsilon_f=(-1)^{k/2}w_N(f)$, this accounts for every power of $p$ and
for the sign used in Theorem~\ref{thm:exact-fricke-trace}.
\end{remark}

\begin{lemma}\label{lem:natural-sum-is-trace}
Let \(k\geq2\) be even and let \(p\nmid N\) be prime.  If
\(H_k^*(N)\) is a normalized Hecke eigenbasis of
\(S_k^{\mathrm{new}}(N)\), then
\begin{equation}\label{eq:natural-trace-normalization}
 \sqrt p\sum_{f\in H_k^*(N)}\lambda_f(p)\epsilon_f
 =(-1)^{k/2}p^{1-k/2}
 \operatorname{Tr}\!\left(T_pW_N\mid S_k^{\mathrm{new}}(N)\right).
\end{equation}
\end{lemma}

\begin{proof}
The eigenvalue of \(T_pW_N\) on the line spanned by \(f\) is
\(a_f(p)w_N(f)\).  Now use
\(a_f(p)=p^{(k-1)/2}\lambda_f(p)\) and \eqref{eq:root-fricke}.
\end{proof}

\begin{lemma}\label{lem:new-full-fricke-trace}
Suppose that \(N\) is squarefree and \(p\nmid N\).  Then
\begin{equation}\label{eq:new-full-fricke-trace}
 \operatorname{Tr}\!\left(T_pW_N\mid S_k^{\mathrm{new}}(N)\right)
 =\operatorname{Tr}\!\left(T_pW_N\mid S_k(N)\right).
\end{equation}
\end{lemma}

\begin{proof}
Corollary~5.14 of \cite{AssafTrace}, in the case \(p\nmid N\), gives
\[
 \operatorname{Tr}\!\left(T_pW_N\mid S_k^{\mathrm{new}}(N)\right)
 =\sum_{\substack{N'\mid N\\ N/N'=\square}}
 \mu\!\left(\sqrt{N/N'}\right)
 \operatorname{Tr}\!\left(T_pW_{N'}\mid S_k(N')\right).
\]
Since \(N\) is squarefree, the condition that \(N/N'\) be a square
forces \(N'=N\).
\end{proof}

For a negative discriminant \(\Delta\), let
\begin{equation}\label{eq:H1-definition}
 H_1(\Delta)=
 \sum_{[Q]\in\mathrm{SL}_2(\mathbb Z)\backslash\mathcal Q_\Delta^+}
 \frac{2}{|\operatorname{Aut}_{\mathrm{SL}_2(\mathbb Z)}(Q)|}.
\end{equation}
Here \(\mathcal Q_\Delta^+\) is the set of positive definite integral
binary quadratic forms of discriminant \(\Delta\).  Thus a generic class
has weight one, while the classes of \(x^2+y^2\) and
\(x^2+xy+y^2\) have weights \(1/2\) and \(1/3\), respectively.  The
sum is zero when \(\Delta\not\equiv0,1\pmod4\).  We also set
\(H_1(0)=-1/12\), although the value at zero does not occur in
Theorem~\ref{thm:exact-fricke-trace}.

Let \(U_j\) denote the Chebyshev polynomial of the second kind,
normalized by
\begin{equation}\label{eq:U-normalization}
 U_j(\cos\theta)=\frac{\sin((j+1)\theta)}{\sin\theta}.
\end{equation}

\begin{theorem}[Exact Atkin--Lehner trace]\label{thm:exact-fricke-trace}
Let \(N\geq2\) be squarefree, let \(p\nmid N\) be prime, and let
\(k\geq2\) be even.  Then
\begin{align}
 &\sqrt p\sum_{f\in H_k^*(N)}\lambda_f(p)\epsilon_f \notag\\
 &\quad=\frac12 H_1(-4Np)\label{eq:exact-fricke-trace}\\
 &+(-1)^{k/2-1}
 \sum_{1\leq r<2\sqrt{p/N}}
 U_{k-2}\!\left(\frac{r\sqrt N}{2\sqrt p}\right)
 H_1(N^2r^2-4Np)\nonumber\\
 &-\delta_{k,2}(p+1).\nonumber
\end{align}
Equivalently, if a prime on the summation sign denotes half weight at
\(r=0\), then the two
class-number terms in \eqref{eq:exact-fricke-trace} equal
\begin{equation}\label{eq:trace-half-weight}
 (-1)^{k/2-1}
 \sum_{0\leq r<2\sqrt{p/N}}^{\prime}
 U_{k-2}\!\left(\frac{r\sqrt N}{2\sqrt p}\right)
 H_1(N^2r^2-4Np).
\end{equation}
\end{theorem}

\begin{proof}
We apply Theorem~4 of \cite{PopaTraceII} with \(Q=N\) and \(n=p\).
In Popa's notation, \(w=k-2\), and the elliptic part of the full-space
trace is
\begin{equation}\label{eq:popa-specialized-elliptic}
 -\frac12
 \sum_{\substack{t^2\leq4Np\\N\mid t}}
 \frac{p_w(t,Np)}{N^{w/2}}
 \sum_{u\mid N}\mu(u)
 H_1\!\left(\frac{t^2-4Np}{u^2}\right),
\end{equation}
where an inadmissible or nonintegral discriminant contributes zero.  This
is also the specialization used in \cite[Section~2]{ZubrilinaMurmurations}.

Put \(t=Nr\).  If \(u>1\), the corresponding term in the inner sum in
\eqref{eq:popa-specialized-elliptic} vanishes.  Indeed, if an odd prime
\(\ell\mid u\), then
\[
 v_\ell(4Np-N^2r^2)=1,
\]
because \(N\) is squarefree and \(p\nmid N\); hence \(u^2\) does not
divide \(4Np-N^2r^2\).  The only remaining possibility is \(u=2\) with
\(N=2N_0\), where \(N_0\) is odd.  In that case
\[
 \frac{4Np-N^2r^2}{4}=N_0(2p-N_0r^2).
\]
This integer is congruent to \(2\pmod4\) when \(r\) is even and to
\(1\pmod4\) when \(r\) is odd.  Its negative is therefore not a
quadratic discriminant, and the class number again vanishes.  Thus only
\(u=1\) remains.

The polynomials in Popa's formula are characterized by
\[
 \sum_{j\geq0}p_j(t,n)x^j=(1-tx+nx^2)^{-1}.
\]
Comparison with the generating series for \(U_j\) gives
\begin{equation}\label{eq:gegenbauer-chebyshev}
 \frac{p_w(Nr,Np)}{N^{w/2}}
 =p^{w/2}U_w\!\left(\frac{r\sqrt N}{2\sqrt p}\right).
\end{equation}

The hyperbolic sum in Popa's formula is empty.  For if \(ad=Np\) and
\(N\mid a+d\), then, for any prime \(\ell\mid N\), squarefreeness of
\(Np\) implies that \(\ell\) divides exactly one of \(a,d\), contrary
to \(\ell\mid a+d\).  The remaining term in that formula is
\[
 \delta_{k,2}\sigma_{1,N}(p)=\delta_{k,2}(p+1).
\]
It follows from \eqref{eq:popa-specialized-elliptic} and
\eqref{eq:gegenbauer-chebyshev} that
\begin{align*}
 &\operatorname{Tr}(T_pW_N\mid S_k(N))\\
 &\quad=-p^{w/2}\bigg\{
 \frac12U_w(0)H_1(-4Np)\\
 &\hspace{28mm}+
 \sum_{1\leq r<2\sqrt{p/N}}
 U_w\!\left(\frac{r\sqrt N}{2\sqrt p}\right)
 H_1(N^2r^2-4Np)\bigg\}\\
 &\hspace{12mm}+\delta_{k,2}(p+1).
\end{align*}
By Lemmas~\ref{lem:natural-sum-is-trace}
and~\ref{lem:new-full-fricke-trace}, we may multiply this identity by
\((-1)^{k/2}p^{-w/2}\).  Since
\[
 U_{k-2}(0)=(-1)^{k/2-1},
\]
the result is \eqref{eq:exact-fricke-trace}.

It remains to justify the strict inequality in the sum.  Equality would
give \(r^2N=4p\).  Every odd prime divisor of \(N\) would then be \(p\),
which is excluded.  Hence \(N\in\{1,2\}\); neither \(r^2=4p\) nor
\(r^2=2p\) is possible for a prime \(p\) with \(p\nmid N\).  Thus the
endpoint never occurs.  Finally, \eqref{eq:trace-half-weight} follows
from the displayed value of \(U_{k-2}(0)\).
\end{proof}

\begin{remark}[The level-one term]\label{rem:level-one-trace}
At level \(N=1\), the hyperbolic sum in Popa's formula is not empty.  The
two factorizations \(p=1\cdot p=p\cdot1\) contribute \(-1\) to
\(\operatorname{Tr}(T_p\mid S_k(1))\).  Consequently the right-hand side
of \eqref{eq:exact-fricke-trace}, with \(N=1\), must be supplemented by
\begin{equation}\label{eq:level-one-correction}
 (-1)^{k/2+1}p^{1-k/2}.
\end{equation}
For \(k=2\), the class-number terms together with
\eqref{eq:level-one-correction} and \(-p-1\) sum to zero, as they must
since \(S_2(1)=0\).  No correction is needed for \(N=2\).
\end{remark}

\begin{remark}[Primes dividing the level]\label{rem:bad-prime-trace}
The condition \(p\nmid N\) is needed in
Lemma~\ref{lem:new-full-fricke-trace}.  If \(N\) is squarefree and
\(p\mid N\), Corollary~5.14 of \cite{AssafTrace} gives instead
\begin{align}\label{eq:bad-prime-new-full}
 \operatorname{Tr}(T_pW_N\mid S_k^{\mathrm{new}}(N))
 ={}&\operatorname{Tr}(T_pW_N\mid S_k(N))\\
 &+(1-p)p^{k/2-1}
 \operatorname{Tr}(W_{N/p}\mid S_k^{\mathrm{new}}(N/p)).\nonumber
\end{align}
Moreover, when \(N=p\), the elliptic sum in Popa's formula has the scalar
endpoint \(r=2\).  Formula \eqref{eq:exact-fricke-trace} is therefore
used only for \(p\nmid N\).
\end{remark}

\section{The weight kernel}\label{sec:weight-kernel}

For $N,p>0$ and $r\geq1$, let
\begin{equation}\label{eq:chebyshev-kernel}
 \Kern_r(N,p)=
 \sum_{\substack{k\geq2\\k\ {\rm even}}}
 V(k/K)F\!\left(\frac{16\pi^2p}{N(k-1)^2}\right)
 (-1)^{k/2-1}
 U_{k-2}\!\left(\frac{r\sqrt N}{2\sqrt p}\right),
\end{equation}
where a summand is taken to be zero if $r\sqrt N>2\sqrt p$.  We also put
\begin{equation}\label{eq:parabolic-kernel}
 \Kern_0(N,p)=
 \sum_{\substack{k\geq2\\k\ {\rm even}}}
 V(k/K)F\!\left(\frac{16\pi^2p}{N(k-1)^2}\right).
\end{equation}

\begin{proposition}[Uniform weight kernel]\label{prop:chebyshev-kernel}
Assume \eqref{eq:basic-scales}.  For every $A>0$,
\begin{equation}\label{eq:chebyshev-kernel-bound}
 |\Kern_r(N,p)|\ll_AK(1+r)^{-A}.
\end{equation}
If $r\sqrt N\leq\sqrt p$, the same estimate holds after applying
$(N\partial_N)^i(p\partial_p)^j$, for any fixed nonnegative integers
$i,j$.
The same estimate without $(1+r)^{-A}$ holds for $\Kern_0$.
\end{proposition}

\begin{proof}
Put $n=k-1=2j+1$, $x=r\sqrt N/(2\sqrt p)$, and
$\theta=\arcsin x$.  From \eqref{eq:U-normalization},
\begin{equation}\label{eq:chebyshev-cosine}
 (-1)^{k/2-1}U_{k-2}(x)
 =\frac{\cos(n\theta)}{\sqrt{1-x^2}}.
\end{equation}
The coefficient
\[
 a(n)=V((n+1)/K)F(16\pi^2p/(Nn^2))
\]
is supported on an interval $n\asymp K$ and satisfies
$a^{(j)}(n)\ll_jK^{-j}$.  Poisson summation on $n=2j+1$ therefore gives,
uniformly for $0\leq\theta\leq\pi/4$,
\begin{equation}\label{eq:odd-poisson}
 \sum_{n\ {\rm odd}}a(n)e^{in\theta}
 \ll_A K(1+K\theta)^{-A}.
\end{equation}
The denominator in \eqref{eq:chebyshev-cosine} is bounded in this range.
Moreover, the support of $a$ gives $\sqrt{p/N}\asymp K$, and hence
$K\theta\asymp Kx\asymp r$ when $\theta\leq\pi/4$.  This proves
\eqref{eq:chebyshev-kernel-bound} without derivatives in the first range.

It remains to check that no loss occurs at $x=1$.  Write
$\delta=\pi/2-\theta$.  For $n=2j+1$,
\begin{equation}\label{eq:endpoint-chebyshev}
 \frac{\cos(n\theta)}{\sqrt{1-x^2}}
 =(-1)^j\frac{\sin(n\delta)}{\sin\delta}.
\end{equation}
Let
\[
 S(u)=\sum_j(-1)^ja(2j+1)e^{i(2j+1)u}.
\]
If $|u|\leq\pi/4$, the frequency in the $j$-sum stays at distance at
least $\pi/2$ from $2\pi\mathbb Z$.  Poisson summation, with arbitrary
many integrations by parts, gives
\begin{equation}\label{eq:alternating-poisson}
 S^{(b)}(u)\ll_{A,b}K^{-A}\qquad(|u|\leq\pi/4).
\end{equation}
The sum of \eqref{eq:endpoint-chebyshev} is
$(S(\delta)-S(-\delta))/(2i\sin\delta)$.  The mean value theorem and
\eqref{eq:alternating-poisson} show that it is $O_A(K^{-A})$, uniformly
down to $\delta=0$.  In this range $r\asymp K$, so this is stronger than
\eqref{eq:chebyshev-kernel-bound}.

Logarithmic differentiation in $N$ or $p$ differentiates $a$ and $x$.
When $x\leq1/2$, each derivative of the phase contributes
$O(n\theta)=O(r)$; this is absorbed by increasing $A$ in
\eqref{eq:odd-poisson}.  This proves the stated derivative bounds in the
range where they will be used.  The assertion for $\Kern_0$ follows
directly from Poisson summation, or from the derivative bounds for $a$.
\end{proof}

Recall that $P=XK^2$.  For $r\geq0$, define
\begin{equation}\label{eq:normalized-kernel-weight}
 \Psi_{X,K,N,r}(u)=\frac1P\Kern_r(N,Pu)
 \sqrt{N(4Pu-r^2N)},
\end{equation}
with the elliptic support understood.

\begin{corollary}\label{cor:uniform-kernel-seminorms}
Fix $\eps_0>0$ and choose $0<\vartheta<\eps_0$.  If
$K\geq X^{\eps_0}$, $N\asymp X$, and $0\leq r\leq X^\vartheta$, then
$\Psi_{X,K,N,r}$ is supported in a fixed compact subinterval of
$(0,\infty)$.  For every $a,j,A\geq0$,
\begin{equation}\label{eq:uniform-kernel-seminorms}
 \sup_{u>0}\left|(N\partial_N)^a\partial_u^j
 \Psi_{X,K,N,r}(u)\right|
 \ll_{a,j,A,F,V}(1+r)^{-A}.
\end{equation}
\end{corollary}

\begin{proof}
On the support of \eqref{eq:normalized-kernel-weight}, one has $u\asymp1$
and
\[
 \frac{r\sqrt N}{2\sqrt{Pu}}\ll\frac rK
 \leq X^{\vartheta-\eps_0}=o(1).
\]
Proposition~\ref{prop:chebyshev-kernel} therefore applies with the required
logarithmic derivatives.  The square-root factor and its derivatives are
$O(XK)$, whereas $P=XK^2$ and
$\Kern_r(N,Pu)\ll_AK(1+r)^{-A}$.  This proves
\eqref{eq:uniform-kernel-seminorms}.
\end{proof}

\section{The class-number expansion}\label{sec:class-number}

For \(p>0\), \(r\geq0\), and \(u>0\), write
\begin{equation}\label{eq:Delta-pr}
 \Delta_{p,r}(u)=-u(4p-r^2u).
\end{equation}
The class-number formula, in the normalization
\eqref{eq:H1-definition}, is
\begin{equation}\label{eq:H1-L-expansion}
 H_1(\Delta)=\frac{\sqrt{|\Delta|}}{\pi}
 \sum_{\substack{f^2\mid\Delta\\
                  \Delta/f^2\equiv0,1\ (4)}}
 \frac1f L(1,\chi_{\Delta/f^2})
 \qquad(\Delta<0).
\end{equation}
The exceptional discriminants \(-3\) and \(-4\) are included by the
weights in \eqref{eq:H1-definition}.  This is the usual decomposition
over quadratic orders followed by Dirichlet's class-number formula; see,
for example, \cite[Chapter~22]{IwaniecKowalski}.

\subsection{Burgess truncation and the outer tails}

We use Burgess's estimate only after fixing the excess over the
quarter-conductor length.

\begin{lemma}[Burgess truncation]\label{lem:burgess-uniform}
Fix \(\eta>0\).  There is \(\delta_B(\eta)>0\) such that, uniformly for
negative discriminants \(d\) with \(|d|\ll Y^2\),
\begin{equation}\label{eq:burgess-uniform}
 L(1,\chi_d)=
 \sum_{m\leq Y^{1/2+\eta}}\frac{\chi_d(m)}m
 +O_\eta(Y^{-\delta_B(\eta)}).
\end{equation}
\end{lemma}

\begin{proof}
Write \(d=d_0c^2\), with \(d_0\) fundamental and \(q=|d_0|\).  Since
\[
 \chi_d(n)=\chi_{d_0}(n)\mathbf1_{(n,c)=1},
\]
inclusion--exclusion gives
\[
 \sum_{n\leq z}\chi_d(n)
 =\sum_{a\mid c}\mu(a)\chi_{d_0}(a)
   \sum_{m\leq z/a}\chi_{d_0}(m).
\]
For every fixed integer \(s\geq2\), Burgess's bound
\cite{BurgessCharacterSums} yields
\[
 \sum_{n\leq z}\chi_d(n)
 \ll_{s,\xi}
 c^\xi z^{1-1/s}q^{(s+1)/(4s^2)+\xi}.
\]
The odd part of \(q\) is squarefree and its \(2\)-part is bounded, so the
quoted form applies.  Partial summation at \(T=Y^{1/2+\eta}\), together
with \(c\ll Y\) and \(q\ll Y^2\), gives
\[
 L(1,\chi_d)-\sum_{m\leq T}\frac{\chi_d(m)}m
 \ll_{s,\xi}
 Y^{1/(2s^2)-\eta/s+3\xi}.
\]
Choose \(s>(2\eta)^{-1}\), and then choose \(\xi>0\) sufficiently
small.
\end{proof}

Set
\[
 \omega_0=\frac12,\qquad \omega_r=1\quad(r\geq1),
\]
and define the elliptic expression
\begin{equation}\label{eq:H-post-class}
 \mathcal H_{X,K}(p)
 =\sum_{r\geq0}\omega_r
  \sum_{N\geq1}^{\mathrm{sf}}W(N/X)
  \Kern_r(N,p)H_1(\Delta_{p,r}(N)).
\end{equation}
The elliptic support is implicit.  For large \(X\),
Theorem~\ref{thm:exact-fricke-trace} shows that
\begin{equation}\label{eq:T-equals-H}
                 \mathcal T_{X,K}(p;F)=\mathcal H_{X,K}(p).
\end{equation}
Indeed \(p>N\), the support of \(W\) excludes \(N=1\), and the support of
\(V(k/K)\) excludes \(k=2\).

Recall that \(Y=XK\), fix \(\eta,\vartheta>0\), and set
\begin{equation}\label{eq:truncation-parameters}
 T=Y^{1/2+\eta},\qquad D=R=X^\vartheta.
\end{equation}
For \(r\leq R\), define
\begin{equation}\label{eq:Phi-pr}
 \Phi_{N,r}(p)=\Kern_r(N,p)\sqrt{N(4p-r^2N)}.
\end{equation}
The truncated expression is
\begin{align}
 \mathcal H^{\mathrm{tr}}_{X,K}(p)
 ={}&\frac1\pi\sum_{0\leq r\leq R}\omega_r
 \sum_{N\geq1}^{\mathrm{sf}}W(N/X)\Phi_{N,r}(p)
 \notag\\
 &\quad\times
 \sum_{\substack{f\leq D\\f^2\mid\Delta_{p,r}(N)\\
        \Delta_{p,r}(N)/f^2\equiv0,1\ (4)}}\frac1f
 \sum_{m\leq T}\frac{\chi_{\Delta_{p,r}(N)/f^2}(m)}m.
                                                        \label{eq:H-truncated}
\end{align}

\begin{lemma}[Outer tails]\label{lem:outer-tails}
Assume \(K\geq X^{\eps_0}\) and choose \(0<\vartheta<\eps_0\).  For every
\(A>0\),
\begin{align}
 \frac{|\mathcal H_{X,K}(p)-\mathcal H^{\mathrm{tr}}_{X,K}(p)|}
      {X^2K^2}
 \ll_{A,\eta}{}&
 R^{-A}
 +(\log Y)^3(D^{-2}+X^{-1})
 +Y^{-\delta_B(\eta)+o(1)}.                 \label{eq:outer-tails}
\end{align}
The same relative estimate holds after summing over
\(p\asymp P=XK^2\) with weight \(\log p\), with \(X^2K^2\) replaced by
\(XP^2=X^3K^4\).
\end{lemma}

\begin{proof}
The standard estimate \(L(1,\chi_d)\ll\log(2|d|)\), together with
\eqref{eq:H1-L-expansion}, gives
\[
 H_1(\Delta)
 \ll |\Delta|^{1/2}\log(2|\Delta|)
       \sum_{f^2\mid\Delta}\frac1f
 \ll |\Delta|^{1/2}\log^2(2|\Delta|).
\]
Proposition~\ref{prop:chebyshev-kernel} therefore gives
\[
 \sum_{r>R}\sum_{N\asymp X}
 |\Kern_r(N,p)H_1(\Delta_{p,r}(N))|
 \ll_A X^2K^2(\log Y)^2R^{-A}.
\]
For an odd \(f\) occurring in \eqref{eq:H1-L-expansion}, squarefreeness
of \(N\) implies \((f,Nr)=1\) and
\[
                  r^2N\equiv4p\pmod{f^2}.
\]
Thus \(N\) lies in one class modulo the odd part of \(f^2\); the factor
at \(2\) changes the count by an absolute factor.  Since
\(L(1,\chi)\ll\log(2|\Delta|)\), the range \(f>D\) contributes at most
\[
 XK(\log Y)^2
 \sum_{f>D}\frac1f\left(\frac X{f^2}+1\right)
 \ll X^2K(\log Y)^3(D^{-2}+X^{-1})
\]
before the kernel is inserted.  The kernel contributes \(O(K)\).

Finally, the Burgess error occurs before the sum over \(m\).  For each
discriminant,
\[
 \sum_{\substack{f\leq D\\f^2\mid\Delta}}\frac1f\ll\log D.
\]
After summing \(N\) and \(r\), its relative contribution is
\(Y^{-\delta_B(\eta)+o(1)}\).  Summation over the primes introduces the
same factor \(P\) in the error and in the main scale.
\end{proof}

\subsection{Completion of the zero frequency}

For \(N,f,m\geq1\) and \(r\geq0\), let \(h_{N,r,f,m}\) be the periodic function of
the prime variable
\begin{equation}\label{eq:h-prime}
 h_{N,r,f,m}(a)=
 \mathbf1_{f^2\mid\Delta_{a,r}(N)}
 \mathbf1_{\Delta_{a,r}(N)/f^2\equiv0,1\ (4)}
 \left(\frac{\Delta_{a,r}(N)/f^2}{m}\right).
\end{equation}
Its period divides \(8f^2m\).  Define its signed mean on the reduced
classes by
\begin{equation}\label{eq:alpha-Nr}
 \alpha_{N,r}(f,m)=
 \frac1{\varphi(8f^2m)}
 \sum_{a\bmod 8f^2m}^{*}h_{N,r,f,m}(a).
\end{equation}
The signed average in \eqref{eq:alpha-Nr} supplies the cancellation
needed for the completion in \(f\) and \(m\).

For a fixed prime \(p\), define instead the periodic function of the
level
\begin{equation}\label{eq:h-level}
\widetilde h_{p,r,f,m}(u)=
\mathbf1_{f^2\mid\Delta_{p,r}(u)}
\mathbf1_{\Delta_{p,r}(u)/f^2\equiv0,1\ (4)}
\left(\frac{\Delta_{p,r}(u)/f^2}{m}\right).
\end{equation}
Write \(f=2^jf_{\mathrm o}\), with \(f_{\mathrm o}\) odd, and put
\begin{align}
 \rho_f(u)
 &=\mathbf1_{(u,f_{\mathrm o})=1}
 \begin{cases}
  1,&j=0,\\
  \mathbf1_{2\nmid u},&j\geq1,
 \end{cases}                                           \label{eq:rho-f}\\
 \widetilde h^{\sharp}_{p,r,f,m}(u)
 &=\rho_f(u)\widetilde h_{p,r,f,m}(u).          \label{eq:h-level-sharp}
\end{align}
In the range \(f<p\) used below, the extra factor records the
coprimality and two-adic restrictions which are automatic on squarefree
levels but not on arbitrary residue classes; this is proved in
Lemma~\ref{lem:squarefree-compatible-restriction}.
If \(h\) is periodic modulo \(q\), its squarefree mean is
\begin{equation}\label{eq:squarefree-periodic-mean}
 \mathfrak m_q(h)=
 \sum_{b\geq1}\frac{\mu(b)}{[b^2,q]}
 \sum_{\substack{a\bmod[b^2,q]\\b^2\mid a}}h(a).
\end{equation}
We use \(\mathfrak m_{q,Z}(h)\) for the same sum restricted to
\(b\leq Z\).

\begin{lemma}[Zero frequency and squarefree density]
\label{lem:zero-squarefree-density}
Let \(h\) be a bounded periodic function modulo \(q\), and let
\(G\in C_c^\infty((0,\infty))\).  After opening
\[
                         \mu^2(N)=\sum_{b^2\mid N}\mu(b),
\]
the total zero Fourier frequency in
\(\sum_{N\geq1}\mu^2(N)h(N)G(N)\) coming from \(b\leq Z\) is exactly
\[
                  \mathfrak m_{q,Z}(h)
                  \int_{\mathbb R}G(t)\,dt.
\]
Moreover,
\begin{equation}\label{eq:squarefree-mean-tail}
 \left|\mathfrak m_q(h)-\mathfrak m_{q,Z}(h)\right|
 \leq \lVert h\rVert_\infty\sum_{b>Z}b^{-2}
 \ll \lVert h\rVert_\infty Z^{-1}.                 
\end{equation}
In particular, \eqref{eq:squarefree-periodic-mean} is absolutely
convergent.  Its value is unchanged if \(q\) is replaced by a
multiple which is also a period of \(h\).  One also has
\begin{equation}\label{eq:squarefree-mean-Cesaro}
 \mathfrak m_q(h)=
 \lim_{Y\to\infty}\frac1Y\sum_{1\leq n\leq Y}\mu^2(n)h(n).
\end{equation}
Consequently, if \(h_1,h_2\) are bounded functions whose periods divide
\(q\), and
\(\mu^2(n)(h_1(n)-h_2(n))=0\) for every \(n\), then
\begin{align}
 \mathfrak m_q(h_1)&=\mathfrak m_q(h_2),
                                                        \label{eq:squarefree-means-agree}\\
 |\mathfrak m_{q,Z}(h_1)-\mathfrak m_{q,Z}(h_2)|
 &\ll (\lVert h_1\rVert_\infty+\lVert h_2\rVert_\infty)Z^{-1}.
                                                        \label{eq:truncated-means-agree}
\end{align}
\end{lemma}

\begin{proof}
For fixed \(b\), put \(L_b=[b^2,q]\).  Decomposition into classes
modulo \(L_b\), followed by Poisson summation, gives, with
\(\widehat G(\tau)=\int_{\mathbb R}G(t)e(-t\tau)\,dt\),
\begin{align*}
 \sum_{\substack{N\in\mathbb Z\\b^2\mid N}}h(N)G(N)
 ={}&\frac1{L_b}
 \sum_{\substack{a\bmod L_b\\b^2\mid a}}h(a)
 \sum_{j\in\mathbb Z}e(ja/L_b)
 \widehat G(j/L_b).
\end{align*}
The term \(j=0\) is
\[
 \frac1{L_b}\sum_{\substack{a\bmod L_b\\b^2\mid a}}h(a)
 \int_{\mathbb R}G(t)\,dt.
\]
Summing with weight \(\mu(b)\) proves the first assertion.  There are
exactly \(L_b/b^2\) classes in the inner sum, so its contribution in
absolute value is at most \(\lVert h\rVert_\infty/b^2\).  This proves
\eqref{eq:squarefree-mean-tail} and absolute convergence.

Finally, if \(q\mid q'\), every admissible class modulo \([b^2,q]\)
has \([b^2,q']/[b^2,q]\) lifts modulo \([b^2,q']\).  Their
multiplicity cancels the enlarged denominator.

To prove \eqref{eq:squarefree-mean-Cesaro}, first truncate the identity
for \(\mu^2\) at \(b\leq B\).  Periodic counting gives
\[
 \lim_{Y\to\infty}\frac1Y\sum_{n\leq Y}h(n)
       \sum_{\substack{b\leq B\\b^2\mid n}}\mu(b)
 =\mathfrak m_{q,B}(h).
\]
The mean absolute value of the omitted part is at most
\[
 \lVert h\rVert_\infty
 \sum_{B<b\leq\sqrt Y}
 \left(\frac1{b^2}+\frac1Y\right)
 \ll \lVert h\rVert_\infty(B^{-1}+Y^{-1/2}).
\]
Letting first \(Y\to\infty\) and then \(B\to\infty\) proves
\eqref{eq:squarefree-mean-Cesaro}.  Apply it to \(h_1-h_2\) to obtain
\eqref{eq:squarefree-means-agree}; combining that equality with
\eqref{eq:squarefree-mean-tail} for \(h_1\) and \(h_2\) proves
\eqref{eq:truncated-means-agree}.
\end{proof}

\begin{lemma}[Squarefree-compatible restriction]
\label{lem:squarefree-compatible-restriction}
Let \(p\) be an odd prime and \(1\leq f<p\).  Then \(8f^2m\) is a
period of both functions in \eqref{eq:h-level} and
\eqref{eq:h-level-sharp}, and
\begin{equation}\label{eq:sharp-squarefree-equality}
 \mu^2(u)\widetilde h^{\sharp}_{p,r,f,m}(u)
 =\mu^2(u)\widetilde h_{p,r,f,m}(u)
 \qquad(u\geq1).
\end{equation}
Consequently, uniformly for \(Z\geq1\),
\begin{align}
 \mathfrak m_{8f^2m}(\widetilde h^{\sharp}_{p,r,f,m})
 &=\mathfrak m_{8f^2m}(\widetilde h_{p,r,f,m}),
                                                        \label{eq:sharp-full-mean}\\
 \mathfrak m_{8f^2m,Z}(\widetilde h^{\sharp}_{p,r,f,m})
 &=\mathfrak m_{8f^2m}(\widetilde h_{p,r,f,m})+O(Z^{-1}).
                                                        \label{eq:sharp-truncated-mean}
\end{align}
\end{lemma}

\begin{proof}
The period assertion follows from \(f_{\mathrm o}\mid f\) and
\eqref{eq:h-level}.  Let \(u\) be squarefree.  If an odd prime
\(\ell\) divided both \(u\) and \(f_{\mathrm o}\), then \(f<p\) would
give \(\ell\ne p\), and
\[
 v_\ell(\Delta_{p,r}(u))
 =v_\ell(u)+v_\ell(4p-r^2u)=1.
\]
This is incompatible with \(f^2\mid\Delta_{p,r}(u)\).

It remains to consider the \(2\)-adic factor when \(j=v_2(f)\geq1\).
Suppose that \(u=2u_0\), with \(u_0\) odd.  Then
\[
 \Delta_{p,r}(u)=4u_0(r^2u_0-2p).
\]
If \(r\) is odd, the second factor is odd and
\[
 \frac{\Delta_{p,r}(u)}4
 =u_0(r^2u_0-2p)\equiv3\pmod4.
\]
If \(r\) is even, that quotient is \(2\pmod4\) and
\(v_2(\Delta_{p,r}(u))=3\).  Thus \(j=1\) is excluded by the
discriminant condition in \eqref{eq:h-level}; division by the remaining
odd square \(f_{\mathrm o}^2\equiv1\pmod4\) does not change either
residue class.  The case \(j\geq2\) is excluded by
\(2^{2j}\nmid\Delta_{p,r}(u)\).  Hence a nonzero
squarefree term with \(j\geq1\) has \(u\) odd.  We have proved
\eqref{eq:sharp-squarefree-equality}.  Equations
\eqref{eq:sharp-full-mean} and \eqref{eq:sharp-truncated-mean} now
follow from Lemma~\ref{lem:zero-squarefree-density} and the bounds
\(\lVert\widetilde h\rVert_\infty,
\lVert\widetilde h^{\sharp}\rVert_\infty\leq1\).
\end{proof}

Put
\begin{align}
 \mathfrak B&=\prod_\ell
 \frac{\ell^4-2\ell^2-\ell+1}{(\ell^2-1)^2},
                                                        \label{eq:B-constant}\\
 \mathfrak A&=\prod_\ell
 \left(1+\frac{\ell}{(\ell+1)^2(\ell-1)}\right).       \label{eq:A-parabolic}
\end{align}

We include the factors from $1/(fm)$ in the local coefficients.  For a
prime $\ell$, put
\[
 \epsilon_\ell(d)=
 \begin{cases}
  \mathbf1_{d\equiv0,1\ (4)},&\ell=2,\\
  1,&\ell\ne2.
 \end{cases}
\]
For $j,t\geq0$ and $M\geq2j+t+3$, define
\begin{align}
 \lambda_{\ell,N,r}(j,t)
 =\frac1{\varphi(\ell^M)}
 \sum_{a\bmod\ell^M}^{*}
 &\mathbf1_{\ell^{2j}\mid\Delta_{a,r}(N)}
 \epsilon_\ell\!\left(
        \frac{\Delta_{a,r}(N)}{\ell^{2j}}\right)
 \left(\frac{\Delta_{a,r}(N)/\ell^{2j}}{\ell^t}\right),
                                                        \label{eq:lambda-local}\\
 a_{\ell,N,r}(j,t)
 &=\ell^{-j-t}\lambda_{\ell,N,r}(j,t).         \label{eq:a-local}
\end{align}

\begin{lemma}[Stabilization and Euler factorization]
\label{lem:local-stabilization}
The value of $\lambda_{\ell,N,r}(j,t)$ is independent of $M$ once
$M\geq2j+t+3$.  Moreover, for all $f,m\geq1$,
\begin{equation}\label{eq:alpha-local-factorization}
 \frac{\alpha_{N,r}(f,m)}{fm}
 =\prod_\ell
 a_{\ell,N,r}\bigl(v_\ell(f),v_\ell(m)\bigr).
\end{equation}
The identity remains valid when $f$ and $m$ are not coprime.
\end{lemma}

\begin{proof}
Fix $\ell,j,t$.  For odd $\ell$, the divisibility condition depends on
$a$ modulo $\ell^{2j}$; if $t>0$, the local symbol additionally depends on
the quotient modulo $\ell$.  Thus the local period divides $\ell^{2j}$
when $t=0$ and $\ell^{2j+1}$ when $t>0$.  At $\ell=2$, the discriminant
indicator and the Kronecker symbol depend only on the quotient modulo eight,
so the local period divides $2^{2j+3}$.  In particular, the integrand is
constant on residue classes modulo $\ell^{2j+t+3}$.  Every unit class modulo
this power has exactly $\ell^{M-(2j+t+3)}$ unit lifts modulo $\ell^M$, and
the same factor occurs in $\varphi(\ell^M)$.  This proves stabilization;
the same lifting argument identifies the local average extracted from the
global modulus $8f^2m$ with this stabilized value.

Now write $j_\ell=v_\ell(f)$ and $t_\ell=v_\ell(m)$.  On reduced residue
classes modulo $8f^2m$, the square-divisor condition, the discriminant
condition, and the Kronecker symbol decompose prime by prime.  At an odd
prime, the square factors contributed by the other prime divisors of $f$
are unit squares and hence do not change the local symbol.  At two, every
odd square is congruent to one modulo eight.  The Chinese remainder theorem
therefore factors the normalized global average as the product of the
stabilized local averages.  A prime dividing both $f$ and $m$ is handled by
the single pair $(j_\ell,t_\ell)$, so no coprimality assumption is used.
Multiplying by $\prod_\ell\ell^{-j_\ell-t_\ell}=1/(fm)$ gives
\eqref{eq:alpha-local-factorization}.
\end{proof}

\begin{proposition}[Signed local completion]
\label{prop:signed-local-completion}
For every \(\xi>0\), uniformly in squarefree \(N\) and \(r\geq1\),
\begin{equation}\label{eq:signed-majorant}
 \sum_{f,m\geq1}\frac{|\alpha_{N,r}(f,m)|}{fm}
 f^{2-\xi}m^{1/2-\xi}\ll_\xi r^\xi.
\end{equation}
Consequently the tails \(f>D\) and \(m>T\) are respectively
\begin{equation}\label{eq:signed-tails}
 O_\xi(r^\xi D^{-2+\xi}),
 \qquad
 O_\xi(r^\xi T^{-1/2+\xi}).
\end{equation}
The fixed-prime zero coefficient satisfies, whenever \(D,T,r<p\),
\begin{align}
 &\sum_{f\leq D}\frac1f\sum_{m\leq T}\frac1m
 \mathfrak m_{8f^2m}(\widetilde h_{p,r,f,m}) \notag\\
 &\qquad=\frac{\mathfrak B\nu(r)}{\zeta(2)}
 +O_\xi\!\left(r^\xi
 \{D^{-2+\xi}+T^{-1/2+\xi}\}
 +p^{-2}\right).
                                                        \label{eq:fixed-zero-series}
\end{align}
For the half-weight channel \(r=0\), one has
\begin{align}
 &\sum_{f,m\geq1}\frac{|\omega_0\alpha_{N,0}(f,m)|}{fm}
 f^{2-\xi}m^{1/2-\xi}\ll_\xi1.                 \label{eq:r0-majorant}
\end{align}
If in addition \(D,T<p\), then
\begin{align}
 &\omega_0\sum_{f\leq D}\frac1f\sum_{m\leq T}\frac1m
 \mathfrak m_{8f^2m}(\widetilde h_{p,0,f,m}) \notag\\
 &\qquad=\frac{\mathfrak A}{2\zeta(2)}
 +O_\xi(D^{-2+\xi}+T^{-1/2+\xi}+p^{-2}).
                                                        \label{eq:r0-fixed-zero}
\end{align}
\end{proposition}

\begin{proof}
By Lemma~\ref{lem:local-stabilization}, it remains to calculate the
stabilized local factors.  We first treat an odd prime.  The following
table is exhaustive.  The three alternatives $\ell\nmid Nr$,
$\ell\mid r$ with $\ell\nmid N$, and $\ell\mid N$ cover every case;
entries not displayed are zero.
\begin{equation}\label{eq:odd-local-table}
\begin{array}{c|c|c|c}
 v_\ell(N)&\text{condition}&(j,t)&a_{\ell,N,r}(j,t)\\ \hline
0&\ell\nmid r&(0,0)&1\\
0&\ell\nmid r&(0,t),\ t\geq1\ {\rm odd}
 &-\dfrac{\ell^{-t}}{\ell-1}\\[2mm]
0&\ell\nmid r&(0,t),\ t\geq2\ {\rm even}
 &\dfrac{\ell-2}{\ell-1}\ell^{-t}\\[2mm]
0&\ell\nmid r&(j,0),\ j\geq1
 &\dfrac{\ell^{1-3j}}{\ell-1}\\[2mm]
0&\ell\nmid r&(j,t),\ j\geq1,\ t\geq2\ {\rm even}
 &\ell^{-3j-t}\\[2mm] \hline
0&\ell\mid r&(0,0)&1\\
0&\ell\mid r&(0,t),\ t\geq2\ {\rm even}&\ell^{-t}\\[2mm] \hline
1&\text{all }r&(0,0)&1.
\end{array}
\end{equation}

Suppose first that \(\ell\nmid Nr\).  The affine polynomial
\(r^2N-4a\) has a unique unit root \(a_0\), and
\[
 \frac1{\ell-1}\sum_{a\in\mathbb F_\ell^\times}
 \left(\frac{N(r^2N-4a)}\ell\right)=-\frac1{\ell-1}.
\]
Indeed, the sum over \(\mathbb F_\ell\) is zero, while the omitted
value at \(a=0\) is one.  Normalized Haar measure on
\(\mathbb Z_\ell^\times\) gives, for \(s\geq1\),
\begin{equation}\label{eq:root-shell-masses}
 \Pr\bigl(v_\ell(a-a_0)\geq s\bigr)=\frac1{\varphi(\ell^s)},
 \qquad
 \Pr\bigl(v_\ell(a-a_0)=s\bigr)=\ell^{-s}.
\end{equation}
When \(j\geq1,t=0\), the whole ball of radius \(\ell^{-2j}\)
contributes.  When \(t>0\), only the shell of exact valuation
\(2j\) contributes.  Its quadratic character has mean zero for odd
\(t\) and equals one for positive even \(t\).  This proves the first
five rows of \eqref{eq:odd-local-table}.

If \(\ell\mid r\) and \(\ell\nmid N\), then \(r^2N-4a\) is a unit,
so \(j=0\).  The odd powers are a constant times
\(\bigl(\frac a\ell\bigr)\) and average to zero; the positive even
powers equal one.  If \(v_\ell(N)=1\), then
\(v_\ell(\Delta_{a,r}(N))=1\) for every unit \(a\), which leaves only
\((j,t)=(0,0)\).  This completes the table.

Summing its rows gives
\begin{align}
 A_{\ell,r}^{(0)}:=\sum_{j,t\geq0}a_{\ell,N,r}(j,t)
 &=1-\frac1{(\ell-1)(\ell^2-1)},
 &&\ell\nmid r,\quad v_\ell(N)=0,              \label{eq:a-unit-N}\\
 A_{\ell,r}^{(0)}&=(1-\ell^{-2})^{-1},
 &&\ell\mid r,\quad v_\ell(N)=0,              \label{eq:a-unit-r}\\
 A_{\ell,r}^{(1)}&=1,
 &&v_\ell(N)=1.                                \label{eq:odd-valuation-one}
\end{align}
For example, when \(\ell\nmid r\), the contribution with \(j=0\)
is \(1-2/((\ell-1)(\ell^2-1))\), whereas
\begin{equation}\label{eq:positive-j-sum}
 \sum_{j\geq1}\left\{
 \frac{\ell^{1-3j}}{\ell-1}
 +\sum_{\substack{t\geq2\\2\mid t}}\ell^{-3j-t}\right\}
 =\frac1{(\ell-1)(\ell^2-1)}.
\end{equation}

We next give the complete calculation at \(2\), with
\(v_2(0)=+\infty\).  The nonzero coefficients are
\begin{equation}\label{eq:two-adic-local-table}
\begin{array}{c|c|c}
 \text{condition}&(j,t)&a_{2,N,r}(j,t)\\ \hline
2\nmid N,\ 2\nmid r&(0,t),\ t\geq0&(-1)^t2^{-t}\\[1mm] \hline
2\nmid N,\ v_2(r)=1&(0,0)&1\\
2\nmid N,\ v_2(r)=1&(1,0)&2^{-2}\\
2\nmid N,\ v_2(r)=1&(j,0),\ j\geq2&2^{2-3j}\\
2\nmid N,\ v_2(r)=1&(j,t),\ j\geq2,\ t\geq2\ {\rm even}
 &2^{1-3j-t}\\[1mm] \hline
2\nmid N,\ v_2(r)\geq2&(0,0)&1\\
2\nmid N,\ v_2(r)\geq2&(1,t),\ t\geq0\ {\rm even}
 &2^{-t-2}\\[1mm] \hline
2\mid N&(0,0)&1.
\end{array}
\end{equation}

If \(N\) and \(r\) are odd, then
\(\Delta_{a,r}(N)\equiv5\pmod8\) for every odd \(a\).  Thus
\(j=0\), the discriminant condition is automatic, and the local
Kronecker symbol equals \(-1\).  This gives the first row.

Suppose that \(N\) is odd and \(r\) is even.  Write
\begin{equation}\label{eq:two-adic-c}
 c=\frac{r^2N}{4},\qquad
 \Delta_{a,r}(N)=4N(c-a).
\end{equation}
If \(v_2(r)\geq2\), then \(c\equiv0\pmod4\) and
\(v_2(\Delta)=2\).  For \(j=1\), the quotient \(N(c-a)\) is an
admissible discriminant precisely when \(a\equiv-N\pmod4\).  Its two
lifts modulo eight give quotient discriminants congruent to one and
five modulo eight.  Odd powers cancel, whereas every positive even
power has mean \(1/2\), giving the third block of
\eqref{eq:two-adic-local-table}.

If \(v_2(r)=1\), then \(c\) is odd.  Put \(w=v_2(c-a)\).  On the odd
classes,
\begin{equation}\label{eq:two-adic-shells}
 \Pr(w=s)=2^{-s}\quad(s\geq1),\qquad
 \Pr(w\geq s)=2^{1-s}.
\end{equation}
For \(j=1\), admissibility is equivalent to \(w\geq2\), giving
\(a_{2,N,r}(1,0)=1/4\).  For \(j\geq2\), the quotient has valuation
zero when \(w=2j-2\), valuation one when \(w=2j-1\), and valuation at
least two when \(w\geq2j\).  The middle shell is inadmissible.  Half
of the first shell has odd quotient congruent to one modulo four; on
that half, the residues one and five modulo eight occur equally often.
Before inserting \(2^{-j-t}\), the means are therefore
\[
 2^{2-2j}\quad(t=0),\qquad
 0\quad(t\ {\rm odd}),\qquad
 2^{1-2j}\quad(t\geq2\ {\rm even}),
\]
which gives the second block of the table.

Finally, if \(N=2N_0\) with \(N_0\) odd, then
\[
 \Delta_{a,r}(N)=4N_0(r^2N_0-2a).
\]
After division by four, the quotient is congruent to three modulo four
when \(r\) is odd, and to two modulo four when \(r\) is even.  Hence
\(j=1\) is inadmissible; for \(j=0\), every term with \(t>0\)
vanishes.  This proves the last row.  Summing
\eqref{eq:two-adic-local-table} yields
\begin{equation}\label{eq:two-adic-unit-sums}
 \sum_{j,t}a_{2,N,r}(j,t)=
 \begin{cases}
  2/3,&2\nmid N,\ 2\nmid r,\\
  4/3,&2\nmid N,\ 2\mid r,\\
  1,&2\mid N.
 \end{cases}
\end{equation}
For \(v_2(r)=1\), for instance, the left side is
\[
 1+\frac14+
 \sum_{j\geq2}\left(2^{2-3j}
 +\sum_{\substack{t\geq2\\2\mid t}}2^{1-3j-t}\right)
 =1+\frac14+\frac1{14}+\frac1{84}=\frac43.
\]
From now on we write \(a_{\ell,r}=A_{\ell,r}^{(0)}\) for the signed
local sum when \(N\) is a unit; at \(\ell=2\), this means \(2/3\) for
odd \(r\) and \(4/3\) for even \(r\).  In every case
\(A_{\ell,r}^{(1)}=1\).

We now average in the squarefree level.  The unnormalized local masses
of \(v_\ell(N)=0,1\) are
\begin{equation}\label{eq:squarefree-local-masses}
                  1-\ell^{-1},\qquad
                  \ell^{-1}-\ell^{-2}.
\end{equation}
Thus
\begin{equation}\label{eq:completed-local-factor}
 C_{\ell,r}:=(1-\ell^{-1})A_{\ell,r}^{(0)}
 + (\ell^{-1}-\ell^{-2})A_{\ell,r}^{(1)}
 =\begin{cases}
 \displaystyle
 \frac{\ell^4-2\ell^2-\ell+1}{\ell^2(\ell^2-1)},
       &\ell\nmid r,\\[3mm]
 \displaystyle
 \frac{\ell^4-\ell^2-\ell+1}{\ell^2(\ell^2-1)},
       &\ell\mid r.
 \end{cases}
\end{equation}
At two, \eqref{eq:two-adic-unit-sums} and the masses \(1/2,1/4\)
give
\[
 C_{2,r}=\frac7{12}\quad(2\nmid r),\qquad
 C_{2,r}=\frac{11}{12}\quad(2\mid r),
\]
which are also the two values in \eqref{eq:completed-local-factor}.
Since \(D_\ell=\ell^4-2\ell^2-\ell+1\),
\begin{equation}\label{eq:completed-product}
 \prod_\ell C_{\ell,r}
 =\prod_\ell\frac{D_\ell}{\ell^2(\ell^2-1)}
  \prod_{\ell\mid r}\left(1+\frac{\ell^2}{D_\ell}\right)
 =\frac{\mathfrak B\nu(r)}{\zeta(2)}.
\end{equation}

We next prove the weighted absolute estimate.  It is enough to take
\(0<\xi\leq1/4\).  From \eqref{eq:odd-local-table}, if
\(\ell\nmid2Nr\),
\begin{align}
 \sum_{j,t\geq0}|a_{\ell,N,r}(j,t)|
 \ell^{(2-\xi)j+(1/2-\xi)t}
 &=1+O_\xi(\ell^{-1-\xi}).                    \label{eq:generic-absolute-factor}
\end{align}
If \(\ell\mid r\) and \(\ell\nmid N\), the corresponding factor is
\begin{equation}\label{eq:ramified-absolute-factor}
 1+\sum_{\substack{t\geq2\\2\mid t}}
 \ell^{-(1/2+\xi)t}
 =1+O_\xi(\ell^{-1-2\xi}),
\end{equation}
and if \(\ell\mid N\) it is one.  The table
\eqref{eq:two-adic-local-table} gives a bounded \(2\)-adic factor.  Hence
\eqref{eq:alpha-local-factorization} and the Euler product prove, in
fact, the stronger bound
\begin{equation}\label{eq:strong-signed-majorant}
 \sum_{f,m\geq1}\frac{|\alpha_{N,r}(f,m)|}{fm}
 f^{2-\xi}m^{1/2-\xi}\ll_\xi1.
\end{equation}
This implies \eqref{eq:signed-majorant}.

Rankin's argument applies to the complement of a rectangle:
\begin{align}
 &\sum_{\substack{f,m\geq1\\f>D\ {\rm or}\ m>T}}
 \frac{|\alpha_{N,r}(f,m)|}{fm}\notag\\
 &\quad\leq
 \left(D^{-2+\xi}+T^{-1/2+\xi}\right)
 \sum_{f,m\geq1}\frac{|\alpha_{N,r}(f,m)|}{fm}
 f^{2-\xi}m^{1/2-\xi}\notag\\
 &\quad\ll_\xi
 r^\xi\left(D^{-2+\xi}+T^{-1/2+\xi}\right).  \label{eq:rectangular-rankin}
\end{align}
This proves \eqref{eq:signed-tails} uniformly in \(N,r\).

It remains to compare the prime-local and level-local averages.  We do
this before any infinite summation.  Let
\({\sf sf}_\ell(n)=\mathbf1_{\ell^2\nmid n}\), and let
\(\mathcal H_{\ell,j,t}(a,n)\) denote the unweighted local integrand
in \eqref{eq:lambda-local}.  For \(M\geq2j+t+3\), set
\begin{equation}\label{eq:finite-completed-coefficient}
 c_{\ell,r}(j,t)
 =\frac{\ell^{-j-t}}{\ell^M\varphi(\ell^M)}
 \sum_{n\bmod\ell^M}{\sf sf}_\ell(n)
 \sum_{a\bmod\ell^M}^{*}\mathcal H_{\ell,j,t}(a,n).
\end{equation}
Reversing these two finite sums and putting \(n=au\) gives
\begin{align}
 &\frac1{\ell^M\varphi(\ell^M)}
 \sum_{a\bmod\ell^M}^{*}\sum_{n\bmod\ell^M}
 {\sf sf}_\ell(n)\mathcal H_{\ell,j,t}(a,n)\notag\\
 &\qquad=\frac1{\ell^M}
 \sum_{u\bmod\ell^M}{\sf sf}_\ell(u)
 \mathcal H_{\ell,j,t}(1,u).                 \label{eq:finite-local-fubini}
\end{align}
Indeed,
\begin{equation}\label{eq:local-square-change}
                    \Delta_{a,r}(au)=a^2\Delta_{1,r}(u).
\end{equation}
The multiplier \(a^2\) is a unit square; at two it is also congruent
to one modulo eight.  It therefore preserves the discriminant
condition and every local symbol.  Thus
\eqref{eq:finite-local-fubini} is an identity of finite sums.

Let \(p\) be an odd prime and suppose \(D,T,r<p\).  For every
\(\ell\ne p\), the substitution \(N=pu\) and
\(\Delta_{p,r}(pu)=p^2\Delta_{1,r}(u)\) show that the fixed-\(p\)
level coefficient equals \eqref{eq:finite-completed-coefficient}.  At
\(\ell=p\), the inequalities \(f,m<p\) force \(j=t=0\); the remaining
factor is the squarefree mass \(1-p^{-2}\).

Introduce the \(p\)-depleted series
\begin{equation}\label{eq:p-depleted-series}
 \Sigma_{p,r}^{(p)}=
 \sum_{\substack{f,m\geq1\\(fm,p)=1}}\frac1{fm}
 \mathfrak m_{8f^2m}(\widetilde h_{p,r,f,m}).
\end{equation}
The local estimates above, combined with
\eqref{eq:squarefree-local-masses}, give
\begin{equation}\label{eq:level-rankin-majorant}
 \sum_{\substack{f,m\geq1\\(fm,p)=1}}
 \frac{|\mathfrak m_{8f^2m}(\widetilde h_{p,r,f,m})|}{fm}
 f^{2-\xi}m^{1/2-\xi}\ll_\xi r^\xi,
\end{equation}
uniformly in \(p\).  Hence
\begin{align}
 &\left|\sum_{f\leq D}\frac1f\sum_{m\leq T}\frac1m
 \mathfrak m_{8f^2m}(\widetilde h_{p,r,f,m})
 -\Sigma_{p,r}^{(p)}\right|\notag\\
 &\qquad\ll_\xi r^\xi
 \left(D^{-2+\xi}+T^{-1/2+\xi}\right).       \label{eq:fixed-level-rectangle}
\end{align}

For \(r\geq1\), the assumption \(r<p\) gives \(p\nmid r\).  Thus
\begin{equation}\label{eq:p-depleted-product}
 \Sigma_{p,r}^{(p)}=(1-p^{-2})\prod_{\ell\ne p}C_{\ell,r}.
\end{equation}
The generic completed factor satisfies
\begin{equation}\label{eq:p-factor-difference}
 C_{p,r}=\frac{p^4-2p^2-p+1}{p^2(p^2-1)},
 \qquad C_{p,r}-(1-p^{-2})=-\frac1{p(p^2-1)}.
\end{equation}
The product over \(\ell\ne p\) is uniformly bounded, so
\[
 \Sigma_{p,r}^{(p)}
 =\frac{\mathfrak B\nu(r)}{\zeta(2)}+O(p^{-3}).
\]
Together with \eqref{eq:fixed-level-rectangle}, this proves
\eqref{eq:fixed-zero-series}.

For \(r=0\), the odd local table is its ramified part.  Thus
\begin{equation}\label{eq:r0-odd-factor}
 C_{\ell,0}=(1-\ell^{-2})
 \left(1+\frac{\ell}{(\ell+1)^2(\ell-1)}\right).
\end{equation}
At two, \eqref{eq:two-adic-unit-sums} gives \(C_{2,0}=11/12\), also
the value of \eqref{eq:r0-odd-factor} at \(\ell=2\).  Hence
\[
                         \prod_\ell C_{\ell,0}
                         =\frac{\mathfrak A}{\zeta(2)}.
\]
The estimate \eqref{eq:ramified-absolute-factor} is summable over all
odd primes, and proves \eqref{eq:r0-majorant}.  Finally,
\[
 C_{p,0}=1-\frac1{p^2(p+1)},\qquad
 C_{p,0}-(1-p^{-2})=\frac1{p(p+1)}.
\]
The rectangular estimate now has no \(r^\xi\) factor.  Multiplication
by \(\omega_0=1/2\) proves \eqref{eq:r0-fixed-zero}.
\end{proof}

For \(r\geq0\), put
\[
 g_r(N)=\sum_{f,m\geq1}\frac{\alpha_{N,r}(f,m)}{fm}.
\]

\begin{proposition}[Squarefree level average]
\label{prop:squarefree-level-average}
For every \(\xi>0\), uniformly for \(r\geq1\), if \(G\) is supported in
\([cX,CX]\), then
\begin{align}
 \sum_{N\geq1}^{\mathrm{sf}}g_r(N)G(N)
 ={}&\frac{\mathfrak B\nu(r)}{\zeta(2)}\int_0^\infty G(t)\,dt \notag\\
 &+O_{\xi,c,C}\left(
 X^{1/2+\xi}r^\xi\mathcal B(G)\right),        \label{eq:squarefree-g-average}
\end{align}
where
\[
 \mathcal B(G)=\|G\|_\infty+\int_{\mathbb R}|G'(t)|\,dt.
\]
For the half-weight channel \(r=0\),
\begin{align}
 \sum_{N\geq1}^{\mathrm{sf}}\omega_0g_0(N)G(N)
 &=\frac{\mathfrak A}{2\zeta(2)}\int_0^\infty G(t)\,dt
 +O_{\xi,c,C}\!\left(X^{1/2+\xi}\mathcal B(G)\right).
                                                        \label{eq:r0-average}
\end{align}
\end{proposition}

\begin{proof}
Absolute signed completion permits the local factors to be
multiplied.  If \(v_\ell(N)=0\), the local factor is \(a_{\ell,r}\)
from \eqref{eq:a-unit-N} or \eqref{eq:a-unit-r}; if
\(v_\ell(N)=1\), it is one.  The same formulas hold at \(\ell=2\):
they give \(2/3\) for odd \(r\) and \(4/3\) for even \(r\).
Consequently, for \(r\geq1\),
\begin{equation}\label{eq:gr-local-product}
 g_r(N)=C_r\prod_{\ell\mid N}a_{\ell,r}^{-1},
 \qquad C_r=\prod_\ell a_{\ell,r}.
\end{equation}
For \(\ell\nmid r\), one has
\(a_{\ell,r}=1+O(\ell^{-3})\), whereas
\(a_{\ell,r}=1+O(\ell^{-2})\) when \(\ell\mid r\).  Thus \(C_r\)
and \(C_r^{-1}\) are bounded uniformly in \(r\).

Put
\[
 \delta_{\ell,r}=a_{\ell,r}^{-1}-1
 =\begin{cases}
 \displaystyle\frac1{\ell(\ell^2-\ell-1)},&\ell\nmid r,\\[2mm]
 -\ell^{-2},&\ell\mid r.
 \end{cases}
\]
The Dirichlet series of \eqref{eq:gr-local-product} is
\begin{equation}\label{eq:gr-Dirichlet}
 \sum_{N\geq1}\frac{\mu^2(N)g_r(N)}{N^s}
 =\zeta(s)\mathcal B_r(s),
\end{equation}
where the Euler factors are explicitly
\begin{align}
 \mathcal B_r(s)
 &=C_r\prod_\ell
 (1-\ell^{-s})(1+a_{\ell,r}^{-1}\ell^{-s}) \notag\\
 &=C_r\prod_\ell
 \left\{1+\delta_{\ell,r}\ell^{-s}
 -(1+\delta_{\ell,r})\ell^{-2s}\right\}.       \label{eq:Br-explicit}
\end{align}
This product is absolutely convergent for \(\Re s>1/2\).  Indeed,
if \(\sigma=1/2+\xi\), then
\begin{align}
 \sum_{d\geq1}\frac{|b_r(d)|}{d^\sigma}
 &\leq |C_r|\prod_\ell
 \left(1+|\delta_{\ell,r}|\ell^{-\sigma}
 +|1+\delta_{\ell,r}|\ell^{-2\sigma}\right)\notag\\
 &\ll_\xi1,                                      \label{eq:br-l1}
\end{align}
where
\[
 \mathcal B_r(s)=\sum_{d\geq1}\frac{b_r(d)}{d^s}.
\]
Here we used \(|\delta_{\ell,r}|\ll\ell^{-2}\) uniformly in \(r\),
and \(\sum_\ell\ell^{-1-2\xi}<\infty\).  In particular,
\eqref{eq:br-l1} is stronger than the asserted \(O_\xi(r^\xi)\)
bound.

At \(s=1\), the local factor including \(a_{\ell,r}\) is
\begin{align*}
 \kappa_{\ell,r}
 &=a_{\ell,r}(1-\ell^{-1})
   (1+a_{\ell,r}^{-1}\ell^{-1})\\
 &=\begin{cases}
 \displaystyle
 \frac{\ell^4-2\ell^2-\ell+1}
      {\ell^2(\ell^2-1)},&\ell\nmid r,\\[3mm]
 \displaystyle
 \frac{\ell^4-\ell^2-\ell+1}
      {\ell^2(\ell^2-1)},&\ell\mid r.
 \end{cases}
\end{align*}
The quotient of the second value by the first is
\[
 1+\frac{\ell^2}{\ell^4-2\ell^2-\ell+1}.
\]
It follows directly from \eqref{eq:B-constant} and the definition of
\(\nu(r)\) that
\begin{equation}\label{eq:Br-at-one}
                  \mathcal B_r(1)
                  =\frac{\mathfrak B\nu(r)}{\zeta(2)}.
\end{equation}

Comparing coefficients in \eqref{eq:gr-Dirichlet} gives
\[
 \sum_{N\geq1}^{\mathrm{sf}}g_r(N)G(N)
 =\sum_{d\leq CX}b_r(d)\sum_{n\geq1}G(dn).
\]
Euler summation for a function of bounded variation yields
\begin{equation}\label{eq:euler-summation-G}
 \sum_{n\geq1}G(dn)
 =\frac1d\int_0^\infty G(t)\,dt+O(\mathcal B(G)).
\end{equation}
Assume first \(0<\xi<1/2\), and write \(\sigma=1/2+\xi\).  By
\eqref{eq:br-l1},
\[
 \sum_{d\leq CX}|b_r(d)|\ll_\xi X^\sigma,
 \qquad
 \sum_{d>CX}\frac{|b_r(d)|}{d}
 \ll_\xi X^{\sigma-1}.
\]
Together with \(\int|G|\ll X\lVert G\rVert_\infty\),
\eqref{eq:euler-summation-G} and \eqref{eq:Br-at-one} prove
\eqref{eq:squarefree-g-average}.  If \(\xi\geq1/2\), apply the
preceding argument with any fixed \(0<\xi'<1/2\) and weaken the
resulting error term.

For \(r=0\), one has at every prime
\[
 a_{\ell,0}=(1-\ell^{-2})^{-1},\qquad
 C_0=\zeta(2),\qquad \delta_{\ell,0}=-\ell^{-2}.
\]
The same coefficient estimate and convolution argument apply, while
the factor at \(s=1\) is
\[
 (1-\ell^{-2})
 \left(1+\frac{\ell}{(\ell+1)^2(\ell-1)}\right).
\]
Their product is \(\mathfrak A/\zeta(2)\).  Multiplication by
\(\omega_0=1/2\) proves \eqref{eq:r0-average}.
\end{proof}

\section{Prime averaging}\label{sec:prime-first}

In this section \eqref{eq:h-prime} is regarded as a function of the
prime.  Put
\begin{equation}\label{eq:qfm}
                         q_{f,m}=8f^2m.
\end{equation}

\begin{lemma}[Support in the prime variable]\label{lem:prime-class-support}
Let $N$ be squarefree, $r\geq0$, and write
$f=2^jf_{\mathrm o}$ and $m=2^tm_{\mathrm o}$ with $f_{\mathrm o}$ and
$m_{\mathrm o}$ odd.  On the reduced classes modulo $q_{f,m}$, the
function $h_{N,r,f,m}$ is supported on at most $16m$ classes.  If the
support is nonempty, then $(f_{\mathrm o},Nr)=1$, and every supported
class satisfies
\begin{equation}\label{eq:prime-class-congruence}
                  4a\equiv r^2N\pmod{f_{\mathrm o}^2}.
\end{equation}
\end{lemma}

\begin{proof}
Let $\ell^e\Vert f_{\mathrm o}$.  If $\ell\mid N$, then
\[
 v_\ell(\Delta_{a,r}(N))=1
\]
for every unit $a$ modulo $\ell$; hence
$\ell^{2e}\nmid\Delta_{a,r}(N)$.  If $\ell\mid r$ and
$\ell\nmid N$, then $\Delta_{a,r}(N)$ is a unit modulo $\ell$.
Thus nonempty support forces $(f_{\mathrm o},Nr)=1$.  The condition
$f_{\mathrm o}^2\mid\Delta_{a,r}(N)$ then fixes the unique unit class
\eqref{eq:prime-class-congruence} modulo $f_{\mathrm o}^2$.  Its lifts to
the odd part of $f^2m$ form at most $m_{\mathrm o}$ classes, also when
$f$ and $m$ have common prime divisors.

It remains to count the two-adic classes.  The first two factors in
\eqref{eq:h-prime} depend only on $a$ modulo $2^{2j+3}$.  If $j=0$,
there are at most four reduced classes modulo eight.  Suppose $j\geq1$.
If $N$ is even, or if $N$ and $r$ are both odd, there is no admissible
class.  Let $N$ be odd and write $r=2r_0$.  If $r_0$ is even, then
$v_2(\Delta_{a,r}(N))=2$ for odd $a$, so only $j=1$ can occur and the
discriminant condition selects at most one class modulo four.  If $r_0$
is odd, divisibility fixes $a$ modulo $2^{2j-2}$; among the four lifts
modulo $2^{2j}$, at most two make the quotient congruent to $0$ or $1$
modulo four.  In all cases there are at most $16$ admissible classes
modulo $2^{2j+3}$.  Lifting to the full two-part
$2^{2j+t+3}$ multiplies this number by at most $2^t$.  The Chinese
remainder theorem now gives at most $16\cdot2^tm_{\mathrm o}=16m$
reduced classes modulo $q_{f,m}$.
\end{proof}

For a smooth function \(\psi\) supported in a fixed compact subinterval
of \((0,\infty)\), put
\begin{equation}\label{eq:Epsi}
 E_\psi(P;q,a)=
 \sum_{\substack{n\geq1\\n\equiv a\ (q)}}\Lambda(n)\psi(n/P)
 -\frac P{\varphi(q)}\int_0^\infty\psi(u)\,du,
 \qquad (a,q)=1.
\end{equation}

\begin{lemma}[Smooth Barban--Davenport--Halberstam]
\label{lem:smooth-BDH}
Let \(I\Subset(0,\infty)\), and let \(\mathscr W\) be a family of
\(C^1\)-functions supported in \(I\) such that
\[
 \sup_{\psi\in\mathscr W}
 \left(\lVert\psi\rVert_\infty+\lVert\psi'\rVert_{L^1}\right)<\infty.
\]
For \(1\leq Q\leq P\) and every \(A>0\),
\begin{equation}\label{eq:smooth-BDH}
 \sup_{\psi\in\mathscr W}
 \sum_{q\leq Q}\sum_{a\bmod q}^{*}
 |E_\psi(P;q,a)|^2
 \ll_{A,\mathscr W}PQ\log P+P^2(\log P)^{-A}.
\end{equation}
\end{lemma}

\begin{proof}
Write
\[
 E(x;q,a)=\sum_{\substack{n\leq x\\n\equiv a\ (q)}}\Lambda(n)
          -\frac{x}{\varphi(q)}.
\]
For each fixed \(C_0>0\), Hooley's form of the
Barban--Davenport--Halberstam theorem
\cite[Theorem~B]{HooleyBDH} gives
\begin{equation}\label{eq:unsmoothed-BDH}
 \sum_{q\leq Q}\sum_{a\bmod q}^{*}|E(x;q,a)|^2
 \ll_{C_0} xQ\log x
 \qquad\left(\frac{x}{(\log x)^{C_0}}\leq Q\leq x\right).
\end{equation}
Given \(A>0\), take \(C_0=A+2\) and put
\(Q_0=x(\log x)^{-C_0}\).  If \(Q<Q_0\), positivity and
\eqref{eq:unsmoothed-BDH} at \(Q_0\) give
\[
 \sum_{q\leq Q}\sum_{a\bmod q}^{*}|E(x;q,a)|^2
 \ll xQ_0\log x\ll x^2(\log x)^{-A}.
\]
It follows that, for every \(1\leq Q\leq x\),
\begin{equation}\label{eq:unsmoothed-BDH-all-Q}
 \sum_{q\leq Q}\sum_{a\bmod q}^{*}|E(x;q,a)|^2
 \ll_A xQ\log x+x^2(\log x)^{-A}.
\end{equation}
The same bound with \(x\asymp P\) holds for \(Q\leq P\).  When
\(Q>x\), split at \(q=x\); for \(x<q\leq Q\), each progression
contains at most one integer up to \(x\), and
\[
 \sum_{a\bmod q}^{*}|E(x;q,a)|^2
 \ll \sum_{n\leq x}\Lambda(n)^2+\frac{x^2}{\varphi(q)}
 \ll x\log x+\frac{x^2}{\varphi(q)}.
\]
Summation over \(x<q\leq Q\) is \(O(xQ\log x)\), as required.

Choose \(0<c<C\) with \(I\subset(c,C)\).  Since the functions vanish
at the endpoints, partial summation gives
\[
 E_\psi(P;q,a)=-\int_c^C\psi'(u)E(Pu;q,a)\,du.
\]
Cauchy's inequality in the integral yields
\[
 |E_\psi(P;q,a)|^2
 \leq\lVert\psi'\rVert_{L^1}
 \int_c^C|\psi'(u)|\,|E(Pu;q,a)|^2\,du.
\]
Summing over \(q,a\) and applying \eqref{eq:unsmoothed-BDH-all-Q}
uniformly for \(u\in[c,C]\) proves \eqref{eq:smooth-BDH}.
\end{proof}

\begin{lemma}[Smooth prime number theorem]
\label{lem:smooth-PNT}
Let \(I\Subset(0,\infty)\).  For every \(B>0\), uniformly for
\(Y\geq2\) and \(\psi\in C_c^1(I)\),
\begin{align}
 \sum_p(\log p)\psi(p/Y)
 ={}&Y\int_0^\infty\psi(u)\,du \notag\\
 &+O_{B,I}\!\left(
 Y(\log Y)^{-B}
 \{\lVert\psi\rVert_\infty+\lVert\psi'\rVert_{L^1}\}\right).
                                                        \label{eq:smooth-PNT}
\end{align}
\end{lemma}

\begin{proof}
This follows by partial summation from
\[
                    \sum_{p\leq x}\log p
                    =x+O_B(x(\log x)^{-B}).
\]
\end{proof}

With \(D,R,T\) as in \eqref{eq:truncation-parameters}, define
\begin{align}
 \mathscr P_{X,K}^{\mathrm{tr}}
 ={}&\sum_{0\leq r\leq R}\omega_r
 \sum_{f\leq D}\frac1f\sum_{m\leq T}\frac1m
 \sum_{N\geq1}^{\mathrm{sf}}W(N/X) \notag\\
 &\quad\times
 \sum_{n\geq1}\Lambda(n)\Psi_{X,K,N,r}(n/P)
 h_{N,r,f,m}(n),                              \label{eq:Ptr}\\
 \mathscr Q_{X,K}^{\mathrm{tr}}
 ={}&P\sum_{0\leq r\leq R}\omega_r
 \sum_{f\leq D}\frac1f\sum_{m\leq T}\frac1m
 \sum_{N\geq1}^{\mathrm{sf}}W(N/X) \notag\\
 &\quad\times
 \left(\int_0^\infty\Psi_{X,K,N,r}(u)\,du\right)
 \alpha_{N,r}(f,m).                            \label{eq:Qtr}
\end{align}

\begin{proposition}[Averaging over primes]
\label{prop:prime-first-replacement}
Fix \(\eps_0,A_0>0\).  Choose \(\eta,\vartheta>0\) sufficiently small
in terms of \(\eps_0\).  Uniformly for
\(X^{\eps_0}\leq K\leq X^{A_0}\), and for every \(B>0\),
\begin{equation}\label{eq:prime-first-replacement}
 \frac{|\mathscr P_{X,K}^{\mathrm{tr}}
       -\mathscr Q_{X,K}^{\mathrm{tr}}|}{XP}
 \ll
 D\left(\frac TP\right)^{1/2}(\log P)^C
 +(\log P)^{-B}.
\end{equation}
Here \(C>0\) is an absolute constant, independent of \(B,X\), and \(K\).
The first term on the right is a negative power of \(X\).
\end{proposition}

\begin{proof}
Remove first every proper prime power from the \(\Lambda\)-sum.  Since
\[
 \sum_{\substack{n\asymp P\\n=\ell^j,\ j\geq2}}\Lambda(n)
 \ll P^{1/2}(\log P)^2,
\]
Corollary~\ref{cor:uniform-kernel-seminorms} and the elementary
\(f\)- and \(m\)-sums show that their total contribution to
\(\mathscr P_{X,K}^{\mathrm{tr}}\) is
\[
                         O\!\left(XP^{1/2}(\log P)^C\right).
\]
This is \(P^{-1/2}(\log P)^C\) relative to \(XP\).  Moreover
\(8D^2T=o(P)\), as verified below.  A prime in the support is therefore
larger than every prime factor of \(q_{f,m}\), so the remaining terms
lie in reduced residue classes.

For fixed \(N,r,f,m\), let \(\omega(a)\) be the value of
\(h_{N,r,f,m}\) on a reduced class and put
\[
 \mathcal R_{N,r,f,m}
 =\sum_{a\bmod q_{f,m}}^{*}\omega(a)
 E_{\Psi_{X,K,N,r}}(P;q_{f,m},a).
\]
By Lemma~\ref{lem:prime-class-support}, $\omega$ is bounded and
supported on $O(m)$ classes, uniformly in $N,r,f,m$.  Hence
\[
 |\mathcal R_{N,r,f,m}|^2
 \ll m\sum_{\substack{a\bmod q_{f,m}\\\omega(a)\ne0}}
 |E_{\Psi_{X,K,N,r}}(P;q_{f,m},a)|^2.
\]
For \(m\asymp M\), Cauchy's inequality gives
\begin{align*}
 \left(\frac1f\sum_{m\asymp M}
       \frac{|\mathcal R_{N,r,f,m}|}{m}\right)^2
 &\ll \frac1{f^2M}
       \sum_{m\asymp M}|\mathcal R_{N,r,f,m}|^2\\
 &\ll \frac1{f^2}
       \sum_{m\asymp M}
       \sum_{\substack{a\bmod 8f^2m\\ \omega(a)\ne0}}
       |E_{\Psi_{X,K,N,r}}(P;8f^2m,a)|^2.
\end{align*}
Enlarging these sparse moduli to all \(q\ll f^2M\) and applying
Lemma~\ref{lem:smooth-BDH}, with \(2A\) in place of \(A\), therefore
gives
\[
 \frac1f\sum_{m\asymp M}\frac{|\mathcal R_{N,r,f,m}|}{m}
 \ll (PM\log P)^{1/2}+\frac{P}{f(\log P)^A}.
\]
Sum the dyadic blocks, \(f\leq D\), and \(N\asymp X\).  The sum over
\(r\) is absolute by Corollary~\ref{cor:uniform-kernel-seminorms}.
This proves \eqref{eq:prime-first-replacement}.

Finally,
\begin{equation}\label{eq:prime-first-exponent}
 D\left(\frac TP\right)^{1/2}
 =X^{-1/4+\vartheta+\eta/2}
  K^{-3/4+\eta/2}.
\end{equation}
For \(\eta<3/2\), the right side decreases with \(K\), so the worst
case is \(K=X^{\eps_0}\).  Taking, for example,
\[
 \eta\leq\frac{\min(\eps_0,1)}{256},
 \qquad
 \vartheta\leq\frac{\min(\eps_0,1)}{512},
\]
gives a power saving.  The modulus condition follows from
\[
 \frac{D^2T}{P}
 =X^{-1/2+2\vartheta+\eta}K^{-3/2+\eta}.
\]
This expression also decreases with $K$ and is a negative power of $X$
for the displayed choices.  Thus $8D^2T=o(P)$, as required when applying
Lemma~\ref{lem:smooth-BDH}.
\end{proof}

Define the explicit pointwise main term
\begin{align}
 \mathcal M_{X,K}(p;F)
 ={}&\frac{\mathfrak c_\varphi}{12}
 \sum_{\substack{k\geq2\\k\ {\rm even}}}(k-1)V(k/K)
 \notag\\
 &\quad\times
 \int_0^\infty tW(t/X) \notag\\
 &\hspace{16mm}
 F\!\left(\frac{16\pi^2p}{t(k-1)^2}\right)
 \M_k(p/t)\,dt.                                  \label{eq:explicit-main}
\end{align}

\begin{proposition}[The zero-frequency integral]
\label{prop:prime-main-identification}
For every \(B>0\), uniformly in
\(X^{\eps_0}\leq K\leq X^{A_0}\),
\begin{align}
 \frac P\pi\mathscr Q_{X,K}^{\mathrm{tr}}
 ={}&\sum_p(\log p)\mathcal M_{X,K}(p;F) \notag\\
 &+O\left(XP^2\left\{
 X^{-\delta}+(\log X)^{-B}\right\}\right)          \label{eq:principal-identification}
\end{align}
for some \(\delta=\delta(\eps_0)>0\).
\end{proposition}

\begin{proof}
Complete first the signed \(f\)- and \(m\)-sums in
\eqref{eq:Qtr}.  Proposition~\ref{prop:signed-local-completion} justifies the
completion and gives tails
\(O(r^\xi D^{-2+\xi})\) and
\(O(r^\xi T^{-1/2+\xi})\).
Proposition~\ref{prop:squarefree-level-average}, applied to
\[
 G_r(t)=W(t/X)\int_0^\infty
 \Phi_{t,r}(Pu)\,\frac{du}{P},
\]
has \(\mathcal B(G_r)\ll_A(1+r)^{-A}\).  After the two outer factors
\(P\) are restored, the sum of its error terms is
\(O(P^2X^{1/2+\xi})\), or \(O(X^{-1/2+\xi})\) relative to the main
scale.  The main local constants are
\(\mathfrak B\nu(r)/\zeta(2)\) for \(r\geq1\) and
\(\mathfrak A/(2\zeta(2))\) for the half-weighted \(r=0\) channel.
Consequently,
\begin{align}
 \frac P\pi\mathscr Q_{X,K}^{\mathrm{tr}}
 ={}&\frac1\pi\int_0^\infty
 \bigg\{
 \frac{\mathfrak B}{\zeta(2)}
 \sum_{r\geq1}\nu(r)
 \int_0^\infty W(t/X)\Phi_{t,r}(p)\,dt \notag\\
 &\hspace{24mm}
 +\frac{\mathfrak A}{2\zeta(2)}
 \int_0^\infty W(t/X)\Phi_{t,0}(p)\,dt
 \bigg\}\,dp
 \notag\\
 &\quad+O(XP^2X^{-\delta}).                    \label{eq:completed-principal}
\end{align}
Here we used \(p=Pu\) and
\(\Psi_{X,K,t,r}(u)=P^{-1}\Phi_{t,r}(Pu)\).

The identities
\[
 \frac{\alpha\mathfrak c_\varphi}{12}
 =\frac{\mathfrak B}{\pi\zeta(2)},\qquad
 \frac{\beta\mathfrak c_\varphi}{12}
 =\frac{\mathfrak A}{\pi\zeta(2)}
\]
are seen directly from the local identities
\begin{align*}
 &\frac{D_\ell}{\ell^4-2\ell^2+\ell}
 \left(1-\frac1{\ell^2+\ell}\right)
 =\left(1-\frac1{\ell^2}\right)
 \frac{D_\ell}{(\ell^2-1)^2},\\
 &\frac{\ell^3+\ell^2-1}{\ell(\ell^2+\ell-1)}
 \left(1-\frac1{\ell^2+\ell}\right)
 =\left(1-\frac1{\ell^2}\right)
 \left(1+\frac{\ell}{(\ell+1)^2(\ell-1)}\right).
\end{align*}
Indeed, multiply over \(\ell\), use
\(\prod_\ell(1-\ell^{-2})=\zeta(2)^{-1}\), and then
\(\zeta(2)=\pi^2/6\).  This proves both displayed identities from
\eqref{eq:cphi}, \eqref{eq:alpha-intro}, \eqref{eq:beta-intro},
\eqref{eq:B-constant}, and \eqref{eq:A-parabolic}.  For \(r=0\),
\(\Phi_{t,0}(p)=2\sqrt{tp}\,\Kern_0(t,p)\), so the second identity
matches the coefficient \(\mathfrak A/(2\pi\zeta(2))\) in
\eqref{eq:completed-principal}.  Proposition~\ref{prop:zubrilina-identity}
therefore identifies \eqref{eq:completed-principal} with the continuous
prime integral of \(\mathcal M_{X,K}(p;F)\).  We keep the \(r\)-sum
truncated while replacing this integral by primes.  Put
\[
 H_r(y)=\int_0^\infty W(t/X)\Phi_{t,r}(y)\,dt,
 \qquad
 \psi_r(u)=\frac{H_r(Pu)}{XP}.
\]
Equivalently,
\[
 \psi_r(u)=\frac1X\int_0^\infty
 W(t/X)\Psi_{X,K,t,r}(u)\,dt.
\]
Corollary~\ref{cor:uniform-kernel-seminorms} shows that, for
\(0\leq r\leq R\), the functions \(\psi_r\) have a common compact
support and
\begin{equation}\label{eq:PNT-r-seminorms}
 \lVert\psi_r\rVert_\infty+
 \lVert\psi_r'\rVert_{L^1}\ll_A(1+r)^{-A}.
\end{equation}
Lemma~\ref{lem:smooth-PNT}, followed by absolute summation over \(r\),
gives
\begin{align}
 \sum_{0\leq r\leq R}c_r
 \left\{\sum_p(\log p)H_r(p)-\int_0^\infty H_r(y)\,dy\right\}
 \ll_B XP^2(\log P)^{-B},                     \label{eq:PNT-r-sum}
\end{align}
where
\[
 c_0=\frac{\mathfrak A}{2\pi\zeta(2)},
 \qquad
 c_r=\frac{\mathfrak B\nu(r)}{\pi\zeta(2)}\quad(r\geq1).
\]
For \(r>R\), no derivative estimate is needed.  The pointwise kernel
bound and Chebyshev's estimate \(\sum_{p\asymp P}\log p\ll P\) give,
after increasing \(A\),
\[
 \sum_{r>R}|c_r|
 \left\{\sum_p(\log p)|H_r(p)|
 +\int_0^\infty|H_r(y)|\,dy\right\}
 \ll_A XP^2R^{-A}.
\]
Combining this with \eqref{eq:PNT-r-sum} replaces the continuous
integral by
\[
                    \sum_p(\log p)\mathcal M_{X,K}(p;F).
\]
Since \(\log P\asymp_{A_0}\log X\), the error is uniform in the
stated range.
\end{proof}

\section{Poisson summation and the cubic range}
\label{sec:exact-conductor}

We keep a prime \(p\asymp P\) fixed and first record a Poisson estimate
uniform in the exponents \(v_\ell(m)\).

For an integral polynomial \(Q\) of degree at most two, put
\[
 \chi_{Q,m}(u)=\left(\frac{Q(u)}m\right),
 \qquad
 q_m=8\prod_{\substack{\ell\mid m\\ \ell\ {\rm odd}}}\ell.
\]
The first function is periodic modulo \(q_m\).  Let \(\vartheta\) be a
function modulo \(8\) with \(\lVert\vartheta\rVert_\infty\leq1\), and set
\[
 \chi_{Q,m}^{\vartheta}(u)=\vartheta(u)\chi_{Q,m}(u).
\]

\begin{lemma}[Fourier mass by exact conductor]
\label{lem:exact-conductor-mass}
Let \(L,S>0\), and let \(G\in C_c^\infty(\mathbb R)\) be supported in an
interval of length \(O(L)\), with
\[
                 \|G^{(j)}\|_\infty\ll_j SL^{-j}
                 \qquad(j\geq0).
\]
For every \(\xi>0\), uniformly in \(Q,m,L\) and \(\vartheta\),
\begin{equation}\label{eq:exact-conductor-mass}
 \sum_{u\in\mathbb Z}G(u)\chi_{Q,m}^{\vartheta}(u)
 =c_{Q,m}^{\vartheta}(0)\int_{\mathbb R}G(t)\,dt
  +O_\xi(Sm^{1/2+\xi}),
\end{equation}
where
\[
 c_{Q,m}^{\vartheta}(0)
 =\frac1{q_m}\sum_{a\bmod q_m}\chi_{Q,m}^{\vartheta}(a).
\]
In particular, a restriction to a collection \(\mathcal R\) of residue
classes modulo \(8\) is obtained by taking
\(\vartheta=\mathbf1_{\mathcal R}\); the zero coefficient then includes
that restriction.
\end{lemma}

\begin{proof}
Write
\[
 c_{Q,m}^{\vartheta}(h)=\frac1{q_m}\sum_{a\bmod q_m}
       \chi_{Q,m}^{\vartheta}(a)e(-ha/q_m).
\]
At an odd prime \(\ell^t\Vert m\), the local function is
\(\chi_\ell(Q)\) when \(t\) is odd and
\(\mathbf1_{\ell\nmid Q}\) when \(t\) is positive and even.  If
\(h_\ell\not\equiv0\pmod\ell\), its normalized Fourier coefficients
satisfy
\begin{equation}\label{eq:local-polynomial-Fourier}
 |c_{\ell,t}(h_\ell)|\leq
 \begin{cases}
  2\ell^{-1/2},&t\ {\rm odd},\\
  2\ell^{-1},&t\ {\rm even}.
 \end{cases}
\end{equation}
For odd $t$, if the reduction of $Q$ is not a constant multiple of a
square, the first line is the Gauss--Weil bound for a polynomial of degree
at most two.  If $Q=cR^2$ with $c\ne0$, then
$\chi_\ell(Q)=\chi_\ell(c)\mathbf1_{R\ne0}$, and the nonzero Fourier
transform is supported by the complement of at most one root; its
normalized size is $O(\ell^{-1})$.  A nonzero constant has zero transform
at nonzero frequency, and the zero polynomial contributes identically
zero.  For positive even $t$, the local function is
$\mathbf1_{\ell\nmid Q}$; subtracting the at most two roots of $Q$ gives
the second line, with the same constant and zero cases included.  At
frequency zero every normalized local coefficient has modulus at most one.

Put
\[
 R_m=\prod_{\substack{\ell\mid m\\ \ell\ {\rm odd}}}\ell.
\]
Write \(A_2(v)\) for the normalized Fourier coefficient at \(v\bmod8\)
of the product of \(\vartheta\) with the \(2\)-adic factor.  Since this
modulus is fixed,
\[
                         \sum_{v\bmod8}|A_2(v)|\leq8.
\]
For every odd \(\ell\mid m\), denote the normalized local coefficients
by \(A_\ell(h_\ell)\).  In a product frequency, put
\[
                         d=\prod_{h_\ell\ne0}\ell.
\]
By the Chinese remainder theorem, the odd part of the resulting
frequency has a unique reduced representative \(a/d\), with
\((a,d)=1\).  Conversely, such an \(a\) determines the nonzero local
frequencies.  Fourier inversion therefore gives
\begin{equation}\label{eq:exact-conductor-expansion}
 \chi_{Q,m}^{\vartheta}(u)
 =\sum_{d\mid R_m}\ \sideset{}{^*}\sum_{a\bmod d}
 C_d(a)e(au/d)
 \sum_{v\bmod8}A_2(v)e(vu/8),
\end{equation}
where for \(d=1\) the starred sum consists of \(a=0\).  By
\eqref{eq:local-polynomial-Fourier}, with \(\xi/3\) in place of
\(\xi\),
\begin{equation}\label{eq:exact-conductor-coefficient}
 |C_d(a)|\leq2^{\omega(d)}d^{-1/2}
 \ll_\xi m^{\xi/3}d^{-1/2}.
\end{equation}

With
\(\widehat G(\tau)=\int_{\mathbb R}G(x)e(-x\tau)\,dx\), Poisson
summation applied to \eqref{eq:exact-conductor-expansion} gives
\begin{align}
 \sum_{u\in\mathbb Z}G(u)\chi_{Q,m}^{\vartheta}(u)
 ={}&\sum_{d\mid R_m}\ \sideset{}{^*}\sum_{a\bmod d}C_d(a)
 \sum_{v\bmod8}A_2(v) \notag\\
 &\hspace{15mm}\times
 \sum_{k\in\mathbb Z}
 \widehat G\!\left(k-\frac v8-\frac ad\right).
                                                        \label{eq:conductor-Poisson}
\end{align}
The unique zero argument is \((d,a,v,k)=(1,0,0,0)\); its coefficient is
\(c_{Q,m}^{\vartheta}(0)\), and it gives the main term in
\eqref{eq:exact-conductor-mass}.  Integration by parts gives
\[
 |\widehat G(\tau)|\ll_A SL(1+L|\tau|)^{-A}.
\]
For the remaining terms and \(A>1\), the map
\[
                 (v,a,k)\longmapsto n=8dk-vd-8a
\]
is injective for fixed \(d\).  Hence, after enlarging the starred sum
to all \(a\bmod d\),
\begin{equation}\label{eq:frequency-lattice}
 L\sum_{v\bmod8}\sum_{a\bmod d}\sum_{k\in\mathbb Z}^{\prime}
 \left(1+L\left|k-\frac v8-\frac ad\right|\right)^{-A}
 \leq L\sum_{n\ne0}
 \left(1+\frac{L|n|}{8d}\right)^{-A}
 \ll_A d,
\end{equation}
where the prime omits a zero argument when it occurs.  Equations
\eqref{eq:exact-conductor-coefficient} and
\eqref{eq:frequency-lattice} show that the contribution with odd
denominator \(d\) is \(O_\xi(Sm^{\xi/3}d^{1/2})\).  Finally,
\[
 \sum_{d\mid R_m}d^{1/2}
 =R_m^{1/2}\prod_{\ell\mid R_m}(1+\ell^{-1/2})
 \ll_\xi m^{1/2+\xi/3}.
\]
Summing over \(d\), and harmlessly enlarging \(2\xi/3\) to \(\xi\),
proves \eqref{eq:exact-conductor-mass}.
\end{proof}

We next isolate the \(2\)-adic condition on the level.

\begin{lemma}[The \(2\)-adic level conditions]
\label{lem:two-adic-level-classes}
Let \(p\) be odd, let \(r\geq0\), and write
\(f=2^jf_{\mathrm o}\), with \(f_{\mathrm o}\) odd.  Put
\[
                    \Delta_{p,r}(N)=N(r^2N-4p),
\]
and use the convention \(v_2(0)=+\infty\).  Once
\(f_{\mathrm o}^2\mid\Delta_{p,r}(N)\) has been imposed, the congruence
\(f_{\mathrm o}^2\equiv1\pmod8\) shows that the \(2\)-adic part of the conditions
\[
 f^2\mid\Delta_{p,r}(N),\qquad
 \Delta_{p,r}(N)/f^2\equiv0,1\pmod4
\]
is equivalent to
\[
 2^{2j}\mid\Delta_{p,r}(N),\qquad
 \Delta_{p,r}(N)/2^{2j}\equiv0,1\pmod4.
\]
For \(j=0\) this local condition holds for every \(N\).  Suppose now
that \(j\geq1\) and \(v_2(N)\leq1\).

\begin{enumerate}
\item If \(2\mid N\), or if \(N\) and \(r\) are both odd, there is no
admissible class.

\item If \(N\) is odd and \(v_2(r)\geq2\), then \(j=1\) is admissible
exactly when
\[
                         N\equiv-p\pmod4,
\]
and \(j\geq2\) is impossible.

\item If \(N\) is odd and \(v_2(r)=1\), write \(r=2r_0\), with \(r_0\)
odd.  For \(j=1\), admissibility is equivalent to
\[
                         N\equiv p\pmod4.
\]
For \(j\geq2\), it is equivalent to
\[
 r_0^2N\equiv p\pmod{2^{2j-2}},
 \qquad
 \frac{r_0^2N-p}{2^{2j-2}}\equiv0\ \text{or}\ N\pmod4.
\]
In this case there is one admissible class modulo \(4\) when \(j=1\),
and two admissible classes modulo \(2^{2j}\) when \(j\geq2\).
\end{enumerate}

In particular, if \(\mathcal C_j(p,r)\) is the set of admissible odd classes
modulo \(2^{2j}\), then
\begin{equation}\label{eq:two-adic-class-mass}
                  \sum_{j\geq1}2^{-j}\#\mathcal C_j(p,r)
                  \leq\frac32.
\end{equation}
\end{lemma}

\begin{proof}
For \(j=0\), one has
\[
 \Delta_{p,r}(N)\equiv
 \begin{cases}
  0\pmod4,&2\mid N\ \text{or}\ 2\mid r,\\
  N^2\equiv1\pmod4,&2\nmid Nr.
 \end{cases}
\]
If \(N=2N_0\), with \(N_0\) odd, then
\[
 \Delta_{p,r}(N)=4N_0(r^2N_0-2p).
\]
For \(r\) odd, the quotient by \(4\) is \(3\pmod4\) and
\(v_2(\Delta_{p,r}(N))=2\); for \(r\) even, it is \(2\pmod4\) and
\(v_2(\Delta_{p,r}(N))=3\).  Thus no \(j\geq1\) is admissible.  If
\(N\) and \(r\) are odd, then \(\Delta_{p,r}(N)\) is odd, which proves
the other assertion in the first case.

Suppose that \(N\) is odd.  If \(r=4s\), then
\[
 \Delta_{p,r}(N)=4N(4s^2N-p),
\]
so \(v_2(\Delta_{p,r}(N))=2\).  For \(j=1\), the quotient is
\(-Np\pmod4\), hence is admissible exactly when
\(N\equiv-p\pmod4\); no \(j\geq2\) is possible.

Finally, let \(r=2r_0\), with \(r_0\) odd.  Then
\[
 \Delta_{p,r}(N)=4N(r_0^2N-p).
\]
For \(j=1\), the quotient is even, so it is admissible precisely when
\(r_0^2N-p\equiv0\pmod4\), equivalently \(N\equiv p\pmod4\).  For
\(j\geq2\), put \(e=2j-2\).  Divisibility is equivalent to
\[
                         r_0^2N\equiv p\pmod{2^e}.
\]
Writing \(y=(r_0^2N-p)/2^e\), the quotient is \(Ny\); since \(N\) is
odd, the discriminant condition is exactly
\(y\equiv0\) or \(N\pmod4\).

The first congruence determines one class \(N_0\pmod{2^e}\).  Its four
lifts modulo \(2^{e+2}\) are \(N_k=N_0+k2^e\), \(0\leq k<4\), and
\[
 \frac{r_0^2N_k-p}{2^e}
 \equiv \frac{r_0^2N_0-p}{2^e}+k\pmod4,
\]
because \(r_0^2\equiv1\pmod4\).  These four values exhaust the residues
modulo \(4\), while \(N_k\pmod4\) is independent of \(k\).  Exactly
two lifts therefore satisfy the discriminant condition.  Finally,
\[
 \frac12+\sum_{j\geq2}\frac2{2^j}=\frac32,
\]
which proves \eqref{eq:two-adic-class-mass}.
\end{proof}

In the applications below \(D<p\) for all sufficiently large \(X\);
indeed,
\[
 D=X^\vartheta<X^{1+2\eps_0}\ll XK^2\asymp p.
\]
Lemma~\ref{lem:squarefree-compatible-restriction} therefore allows us,
for each \(f\leq D\), to replace \(\widetilde h_{p,r,f,m}\) by
\(\widetilde h^{\sharp}_{p,r,f,m}\) in the squarefree level sum.  We
then open the squarefree condition,
\begin{equation}\label{eq:squarefree-sieve}
                         \mu^2(N)=\sum_{b^2\mid N}\mu(b).
\end{equation}
On the support of the restricted function,
\((f_{\mathrm o},Nr)=1\) and
\[
                         r^2N\equiv4p\pmod{f_{\mathrm o}^2}.
\]
Indeed, the first coprimality in \(N\) is part of \(\rho_f\); if an
odd prime \(\ell\mid(f_{\mathrm o},r)\), then
\(\Delta_{p,r}(N)\equiv-4pN\not\equiv0\pmod\ell\).
Moreover, \(b^2\mid N\) implies \((b,f)=1\): the odd part follows from
\((N,f_{\mathrm o})=1\), while for \(j\geq1\) the definition of
\(\rho_f\) forces \(N\), and hence \(b\), to be odd.  For fixed
\(b,f,r\), the \(2\)-adic conditions are now given exactly by
Lemma~\ref{lem:two-adic-level-classes}.  When \(j=0\), no parity
restriction is added.  When \(j\geq1\), the factor \(\rho_f\) restricts
\(N\) to the odd classes.  There are no such classes if \(r\) is odd;
if \(4\mid r\), then among \(j\geq1\) only \(j=1\) can occur, with
\(N\equiv-p\pmod4\); and if \(v_2(r)=1\), there is one class for
\(j=1\) and two classes for each \(j\geq2\), as described in the lemma.
In particular, their total number after multiplication by the weight
\(2^{-j}\) from \(1/f\) is bounded by
\eqref{eq:two-adic-class-mass}.

Combining these classes with
\(r^2N\equiv4p\pmod{f_{\mathrm o}^2}\) by the Chinese remainder theorem,
and using \((b,f)=1\), gives one progression when $j=0$, at most one
progression when $j=1$, and at most two progressions for each $j\geq2$.
Their total multiplicity after the factor $2^{-j}$ is bounded by
\eqref{eq:two-adic-class-mass}.  Each progression has the form
\begin{equation}\label{eq:level-progression-after-sieve}
                 N=b^2(n_{\mathfrak c}+f^2u),
\end{equation}
with \(n_{\mathfrak c}\) fixed modulo \(f^2\), and
\begin{equation}\label{eq:quotient-polynomial}
 Q_{b,f,r,p}(u)=
 \frac{\Delta_{p,r}(b^2(n_{\mathfrak c}+f^2u))}{f^2}
\end{equation}
is an integral polynomial of degree at most two.
Any remaining factor modulo eight, including the \(2\)-adic factor of
the Kronecker symbol, is included in \(\vartheta\) in
Lemma~\ref{lem:exact-conductor-mass}.  On every resulting progression
the step is exactly \(b^2f^2\), so its
length is \(L\asymp X/(b^2f^2)\).

\begin{lemma}[Exact lifting of the zero coefficient]
\label{lem:exact-zero-lifting}
Let $b,f,q_0\geq1$, put $q=f^2q_0$, $M=b^2q$, and
$L=[b^2,q]$.  Let $h$ be periodic modulo $q$, and let
$\mathcal C\subset\mathbb Z/f^2\mathbb Z$ have the property that
$h(b^2n)=0$ whenever $n\bmod f^2$ does not belong to $\mathcal C$.  Then
\begin{equation}\label{eq:general-zero-lifting}
 \frac1{b^2f^2q_0}
 \sum_{c\in\mathcal C}\sum_{v\bmod q_0}
 h\bigl(b^2(c+f^2v)\bigr)
 =\frac1L\sum_{\substack{a\bmod L\\b^2\mid a}}h(a).
\end{equation}
No coprimality between $b$ and $q_0$ is required.
\end{lemma}

\begin{proof}
The map $(c,v)\mapsto c+f^2v$ is a bijection from
$(\mathbb Z/f^2\mathbb Z)\times(\mathbb Z/q_0\mathbb Z)$ onto
$\mathbb Z/q\mathbb Z$: reduction modulo $f^2$ recovers $c$, and division
of the difference by $f^2$ recovers $v$ modulo $q_0$.  By the support
assumption, restricting $c$ to $\mathcal C$ does not change the weighted
sum.  Multiplication by $b^2$ is then a bijection from classes modulo $q$
to classes modulo $M=b^2q$ which are divisible by $b^2$.  Hence the left
side of \eqref{eq:general-zero-lifting} equals
\[
 \frac1M\sum_{\substack{a\bmod M\\b^2\mid a}}h(a).
\]
Since $L\mid M$ and both $b^2$ and $q$ divide $L$, every admissible class
modulo $L$ has exactly $M/L$ admissible lifts modulo $M$, and $h$ is constant
on those lifts.  Dividing by $M$ proves \eqref{eq:general-zero-lifting}.
\end{proof}

Let \(\mathcal H^{0}_{X,K}(p;D,T,R,Z)\) and
\(\mathcal H^{\ne0}_{X,K}(p;D,T,R,Z)\) be, respectively, the zero and
nonzero terms obtained by applying
Lemma~\ref{lem:exact-conductor-mass} to every summand of
\eqref{eq:H-truncated}, after the preceding replacement and
\eqref{eq:squarefree-sieve}, with
\(b\leq Z\).  Write \(\mathcal H^{b>Z}_{X,K}(p)\) for the omitted sieve
range.  These three expressions give an exact decomposition of the
truncated sum.

\begin{proposition}[The nonzero frequencies]
\label{prop:global-nonzero}
Let \(D=R=X^\vartheta\), \(T=(XK)^{1/2+\eta}\), and
\(1\leq Z\leq X^{1/4}\).  Uniformly for
\(X^{\eps_0}\leq K\leq X^{3-\eps_0}\) and \(p\asymp XK^2\),
\begin{align}
 \frac{|\mathcal H^{\ne0}_{X,K}(p;D,T,R,Z)|}{X^2K^2}
 &\ll_\xi (XK)^\xi\frac{ZT^{1/2}}X,
                                                        \label{eq:global-nonzero}\\
 \frac{|\mathcal H^{b>Z}_{X,K}(p)|}{X^2K^2}
 &\ll_\xi (XK)^\xi Z^{-1}.             \label{eq:squarefree-tail}
\end{align}
\end{proposition}

\begin{proof}
In \eqref{eq:level-progression-after-sieve}, the smooth weight has
length \(L\asymp X/(b^2f^2)\).
Proposition~\ref{prop:chebyshev-kernel} gives, for every \(j,A\geq0\),
\[
 \left|\frac{d^j}{du^j}
 \left\{W(N/X)\Phi_{N,r}(p)\right\}\right|
 \ll_{j,A}P(1+r)^{-A}L^{-j}.
\]
On each retained progression \(\rho_f\) is identically one.  Thus the
periodic factor remains the same modulo-eight cutoff times
\(\chi_{Q_{b,f,r,p},m}\), and it remains periodic modulo \(q_m\).
The exact-conductor decomposition in
Lemma~\ref{lem:exact-conductor-mass} therefore bounds the nonzero part
for fixed \(b,f,m,r\) by
\(P(1+r)^{-A}m^{1/2+\xi}\).  Restoring the weights in
\eqref{eq:H-truncated} and summing gives
\begin{align*}
 |\mathcal H^{\ne0}_{X,K}(p;D,T,R,Z)|
 &\ll_\xi P
 \sum_{b\leq Z}\sum_{f\leq D}\frac1f
 \sum_{m\leq T}m^{-1/2+\xi}
 \sum_{r\leq R}(1+r)^{-A}\\
 &\ll_\xi PZ(XK)^\xi T^{1/2}.
\end{align*}
Since \(X^2K^2=XP\), this proves \eqref{eq:global-nonzero}.

For the omitted sieve range, \(|\rho_f|\leq1\), and
\[
 \sum_{N\asymp X}\left|
 \sum_{\substack{b^2\mid N\\b>Z}}\mu(b)\right|
 \ll X\sum_{b>Z}b^{-2}\ll X/Z.
\]
Both the complete class-number factor and its truncation satisfy
\begin{multline*}
 |\Kern_r(N,p)|\left\{
 H_1(\Delta_{p,r}(N))
 +\sqrt{|\Delta_{p,r}(N)|}
   \sum_{f\leq D}\frac1f\sum_{m\leq T}\frac1m\right\}\\
 \ll_{\xi,A}P(1+r)^{-A}(XK)^\xi.
\end{multline*}
Summing \(r\) proves
\eqref{eq:squarefree-tail}.
\end{proof}

\begin{proposition}[The zero frequency]
\label{prop:fixed-prime-zero}
With the same notation, there is
\(\delta_0=\delta_0(\eps_0)>0\) such that, for every sufficiently small
fixed \(\xi>0\),
\begin{equation}\label{eq:fixed-prime-zero}
 \mathcal H^{0}_{X,K}(p;D,T,R,Z)
 =\mathcal M_{X,K}(p;F)
 +O_{F,V,W,\eps_0}\!\left(
 XP\{X^{-\delta_0}+(XK)^\xi Z^{-1}\}\right).
\end{equation}
\end{proposition}

\begin{proof}
Put
\[
                 G_{p,r}(t)=W(t/X)\Phi_{t,r}(p).
\]
We first identify exactly the zero term produced by the progressions
in \eqref{eq:level-progression-after-sieve}.  Put
\[
 q_m=8\prod_{\substack{\ell\mid m\\\ell\ {\rm odd}}}\ell,
 \qquad q_{f,m}^{*}=f^2q_m.
\]
At an odd prime dividing $m$, the Jacobi symbol depends only on the
numerator modulo that prime, while its two-adic component depends only
modulo eight.  Consequently both \(\widetilde h_{p,r,f,m}\) and
\(\widetilde h^{\sharp}_{p,r,f,m}\) are periodic modulo
\(q_{f,m}^{*}\); enlarging this period to \(8f^2m\) does not change
their squarefree means by Lemma~\ref{lem:zero-squarefree-density}.

Fix \(b,f,m\).  After the two-adic refinement, the admissible level
classes were written as
\(N=b^2(n_{\mathfrak c}+f^2u)\).  On such a class the character is a
periodic function modulo \(q_m\), namely
\[
 \vartheta_{\mathfrak c}(u)=
 \widetilde h^{\sharp}_{p,r,f,m}
 \bigl(b^2(n_{\mathfrak c}+f^2u)\bigr).
\]
Its zero term is
\begin{equation}\label{eq:one-progression-zero}
 \frac1{b^2f^2q_m}
 \sum_{v\bmod q_m}\vartheta_{\mathfrak c}(v)
 \int_{\mathbb R}G_{p,r}(t)\,dt.
\end{equation}
Indeed, \(dN=b^2f^2\,du\), and decomposing further modulo \(q_m\)
contributes the factor \(q_m^{-1}\).

Let $M_b=b^2f^2q_m$ and $L_b=[b^2,q_{f,m}^{*}]$.  Apply
Lemma~\ref{lem:exact-zero-lifting} with $q_0=q_m$,
$q=q_{f,m}^{*}$, $h=\widetilde h^{\sharp}_{p,r,f,m}$, and
$\mathcal C=\{n_{\mathfrak c}:\mathfrak c\text{ admissible}\}$.  The
support condition in that lemma is
exactly the square-divisor and discriminant classification used to obtain
\eqref{eq:level-progression-after-sieve}.  We obtain
\begin{align}
 &\frac1{M_b}
 \sum_{\substack{a\bmod M_b\\b^2\mid a}}
       \widetilde h^{\sharp}_{p,r,f,m}(a) \notag\\
 &\hspace{20mm}=
 \frac1{L_b}
 \sum_{\substack{a\bmod L_b\\b^2\mid a}}
       \widetilde h^{\sharp}_{p,r,f,m}(a).      \label{eq:zero-class-lifting}
\end{align}
This identity includes the modulus-eight discriminant classes and does not
require $(b,m)=1$.  Summing \eqref{eq:one-progression-zero} first over the classes
and then over \(b\leq Z\) gives the exact identity
\[
 \mathcal Z_{p,r}^{(Z)}(D,T)=
 \sum_{f\leq D}\frac1f\sum_{m\leq T}\frac1m
 \mathfrak m_{8f^2m,Z}(\widetilde h^{\sharp}_{p,r,f,m}).
\]
\begin{align}
 \mathcal H^{0}_{X,K}(p;D,T,R,Z)
 ={}&\frac1\pi\sum_{0\leq r\leq R}\omega_r
 \mathcal Z_{p,r}^{(Z)}(D,T)
 \int_0^\infty G_{p,r}(t)\,dt.                \label{eq:zero-density-exact}
\end{align}

Define
\[
 \mathcal Z_{p,r}(D,T)=
 \sum_{f\leq D}\frac1f\sum_{m\leq T}\frac1m
 \mathfrak m_{8f^2m}(\widetilde h_{p,r,f,m}).
\]
Lemma~\ref{lem:squarefree-compatible-restriction} gives
\[
 \mathfrak m_{8f^2m,Z}(\widetilde h^{\sharp}_{p,r,f,m})
 =\mathfrak m_{8f^2m}(\widetilde h_{p,r,f,m})+O(Z^{-1}).
\]
Since
\[
 \sum_{f\leq D}\frac1f\sum_{m\leq T}\frac1m
 \ll\log(2D)\log(2T)\ll_\xi(XK)^\xi
\]
and \(\int|G_{p,r}(t)|\,dt\ll_A XP(1+r)^{-A}\),
\eqref{eq:zero-density-exact} becomes
\begin{align}
 \mathcal H^{0}_{X,K}(p;D,T,R,Z)
 ={}&\frac1\pi\sum_{0\leq r\leq R}\omega_r
 \mathcal Z_{p,r}(D,T)
 \int_0^\infty G_{p,r}(t)\,dt \notag\\
 &+O_\xi\!\left(XP(XK)^\xi Z^{-1}\right).       \label{eq:zero-frequency-bridge}
\end{align}
Since \(D,T,R<p\) for large \(X\), equations
\eqref{eq:fixed-zero-series} and \eqref{eq:r0-fixed-zero}, applied
before summing \(r\), give
\begin{align}
 \mathcal H^{0}_{X,K}(p;D,T,R,Z)
 ={}&\frac{\mathfrak B}{\pi\zeta(2)}
 \sum_{r\geq1}\nu(r)\int_0^\infty G_{p,r}(t)\,dt \notag\\
 &+\frac{\mathfrak A}{2\pi\zeta(2)}
 \int_0^\infty G_{p,0}(t)\,dt
 +O(XP E_0),                                    \label{eq:zero-before-M}
\end{align}
where, for every fixed \(\xi>0\),
\[
 E_0\ll_\xi
 R^{-A}+D^{-2+\xi}+T^{-1/2+\xi}+p^{-2}
 +(XK)^\xi Z^{-1}.
\]
The first error uses the rapid decay of the weight kernel, and the last
one restores the range \(b>Z\).

The identities
\[
 \frac{\alpha\mathfrak c_\varphi}{12}
 =\frac{\mathfrak B}{\pi\zeta(2)},\qquad
 \frac{\beta\mathfrak c_\varphi}{12}
 =\frac{\mathfrak A}{\pi\zeta(2)}
\]
hold.  Since \(\Phi_{t,0}(p)=2\sqrt{tp}\,\Kern_0(t,p)\),
the second identity gives
\[
 \frac{\mathfrak A}{2\pi\zeta(2)}\Phi_{t,0}(p)
 =\frac{\beta\mathfrak c_\varphi}{12}
   \sqrt{tp}\,\Kern_0(t,p).
\]
Proposition~\ref{prop:zubrilina-identity}, including its parabolic term,
identifies the first two lines of \eqref{eq:zero-before-M} with
\(\mathcal M_{X,K}(p;F)\).  The Burgess remainder and every term in
\(E_0\), apart from the displayed sieve tail, are negative powers of
\(X\) for the parameters chosen in Section~\ref{sec:pointwise-proof}.
Their total contribution is \(O(X^{-\delta_0})\) for some
\(\delta_0>0\).
\end{proof}

\section{Proof of the pointwise formula}\label{sec:pointwise-proof}

\begin{proof}[Proof of Theorem~\ref{thm:pointwise-main}]
Fix \(0<\eps_0<3/2\), put \(\eps_*=\min(\eps_0,1)\), and choose
\[
 \eta=\frac{\eps_*}{256},\qquad
 \vartheta=\frac{\eps_*}{512},\qquad
 D=R=X^\vartheta.
\]
Lemma~\ref{lem:outer-tails} reduces
\(\mathcal H_{X,K}(p)\) to \(\mathcal H^{\mathrm{tr}}_{X,K}(p)\)
with a power-saving error.

The zero Poisson frequency is evaluated by
Proposition~\ref{prop:fixed-prime-zero}.

It remains to estimate the nonzero frequencies.  Write \(K=X^\rho\)
and choose
\[
 Z=X^{\beta_0},\qquad
 \beta_0=\frac{\eps_*}{16},\qquad
 \sigma=\frac{\eps_*}{64},\qquad
 \xi=\frac{\sigma}{4}.
\]
Since \(T=(XK)^{1/2+\eta}\),
\begin{equation}\label{eq:T-half-exponent}
 \frac{T^{1/2}}X
 =X^{(\rho-3)/4+\eta(1+\rho)/2}.
\end{equation}
For \(\rho\leq3-\eps_0\leq3-\eps_*\), the exponent on the right is at most
\(-\eps_*/4+2\eta\).  Moreover
\((XK)^\xi\leq X^{4\xi}=X^\sigma\), because $1+\rho\leq4$.  Hence the
sieve tail in \eqref{eq:squarefree-tail} is at most
\[
                         X^{\sigma-\beta_0}=X^{-3\eps_*/64}.
\]
For the nonzero frequencies, \eqref{eq:global-nonzero} and
\eqref{eq:T-half-exponent} give the exponent
\[
 \beta_0+\sigma-\frac{\eps_*}{4}+2\eta
 =-\frac{21\eps_*}{128}.
\]
Also $\beta_0\leq1/16$, so $Z\leq X^{1/4}$ as required in
Proposition~\ref{prop:global-nonzero}.  Thus both terms are fixed negative
powers of $X$.
Combining these estimates with \eqref{eq:fixed-prime-zero},
Lemma~\ref{lem:outer-tails}, and \eqref{eq:T-equals-H} proves
\eqref{eq:pointwise-main}.  One may take
any
\[
 \delta<\min\left\{
 \frac{3\eps_*}{64},\,\vartheta,\,
 \frac12\delta_B(\eta),\,\delta_0\right\}.
\]
\end{proof}

\section{Neumann summation}\label{sec:neumann}

Let $0<c_1<c_2$ and $S>0$.  Let $w$ be smooth, supported in
$[c_1K,c_2K]$, and suppose that
$|w^{(j)}(t)|\ll_j S K^{-j}$.
Set
\begin{equation}\label{eq:Tw}
 T_w(x)=\sum_{\substack{n\geq1\\n\ {\rm odd}}}n w(n)J_n(x).
\end{equation}

\begin{proposition}[Smoothed Neumann summation]\label{prop:neumann}
For every $A>0$ and $0<x\leq K^{5/2}$,
\begin{equation}\label{eq:neumann-smoothed}
 T_w(x)=\frac x2w(x)
 +O_A\!\left(S\frac{x}{K^2}\1_{x\asymp K}+SK^{-A}\right).
\end{equation}
For $x>K^{5/2}$,
\begin{equation}\label{eq:neumann-large-x}
 T_w(x)\ll_A SK(K/x)^A.
\end{equation}
\end{proposition}

\begin{proof}
The recurrence
$J_{n-1}(x)+J_{n+1}(x)=2nJ_n(x)/x$ and the Jacobi--Anger expansion give
\begin{equation}\label{eq:neumann-exact}
 \sum_{n\geq1,\ n\ {\rm odd}}nJ_n(x)
 =\frac x2\left(J_0(x)+2\sum_{m\geq1}J_{2m}(x)\right)=\frac x2.
\end{equation}
We need a localized version of this identity.

Set $W_1(t)=tw(t)$.  In this section only, use the angular Fourier transform
$\widehat W_1(\xi)=\int_{\R}W_1(t)e^{-it\xi}\,dt$.  The assumptions on
$w$ imply
\begin{equation}\label{eq:W1-fourier}
 |\widehat W_1^{(b)}(\xi)|
 \ll_{A,b}SK^{b+2}(1+K|\xi|)^{-A}.
\end{equation}
The integral formula for $J_n$ and Poisson summation on the odd integers
give
\begin{equation}\label{eq:Tw-poisson}
\begin{split}
 T_w(x)=\frac1{4\pi}\int_{-\pi}^{\pi}e^{-ix\sin\theta}
 \sum_{j\in\mathbb Z}\bigl\{&
 \widehat W_1(2\pi j-\theta)\\
 &-\widehat W_1(2\pi j-\theta-\pi)\bigr\}\,d\theta.
\end{split}
\end{equation}
By \eqref{eq:W1-fourier}, only neighborhoods of $0$ and $\pm\pi$ of
length $K^{-1+\eta_0}$ contribute, up to $O_A(SK^{-A})$.  Combining the
two half-neighborhoods at $\pm\pi$ yields
\begin{equation}\label{eq:Tw-local-fourier}
 T_w(x)=\frac1{4\pi}\{I(x)-I(-x)\}+O_A(SK^{-A}),
\end{equation}
where
\[
 I(\pm x)=\int_{\R}\chi(\theta)e^{\mp ix\sin\theta}
                  \widehat W_1(-\theta)\,d\theta
\]
and $\chi$ is supported in $|\theta|\leq2K^{-1+\eta_0}$ and equals one
on the smaller contributing neighborhood.

Suppose first that $x\leq K^{5/2}$.  On the support of $\chi$ write
\(
 \sin\theta=\theta+u(\theta)
\), where
\begin{equation}\label{eq:cubic-u}
 u(\theta)=-\frac{\theta^3}{6}
 \left(1-\frac{\theta^2}{20}+O(\theta^4)\right).
\end{equation}
Choose $\eta_0<1/12$.  Taylor expansion of $e^{\mp ixu(\theta)}$ to an
arbitrarily large fixed order is uniform, since
$x|u(\theta)|\ll K^{-1/4}$.  Every retained term is a linear combination
of
\[
 x^m\int_{\R}\theta^{3m+2l}\widehat W_1(-\theta)
 e^{\mp ix\theta}\,d\theta
 =2\pi c_{m,l}^{\pm}x^mW_1^{(3m+2l)}(\pm x),
\]
up to $O_A(SK^{-A})$.  The term $m=l=0$ is $2\pi W_1(\pm x)$.  Since
$x>0$ and $W_1$ is supported in $[c_1K,c_2K]$, only $W_1(x)$ occurs.
For $m\geq1$ the derivative bounds give
\[
 x^m|W_1^{(3m+2l)}(x)|
 \ll S\frac{x}{K^2}
       \left(\frac{x}{K^3}\right)^{m-1}K^{-2l}.
\]
The double sum is $O(Sx/K^2)$; it vanishes outside a fixed enlargement of
the support of $w$.  Equation \eqref{eq:Tw-local-fourier} proves
\eqref{eq:neumann-smoothed}.

If $x>K^{5/2}$, integrate repeatedly in \eqref{eq:Tw-local-fourier} against
the phase $x\sin\theta$.  Its derivative is comparable with $x$ on the
support of $\chi$, while differentiation of the amplitude costs $K$.
This proves \eqref{eq:neumann-large-x}.
\end{proof}

\begin{proposition}[Bessel series to atomic masses]\label{prop:bessel-to-atoms}
Fix $\eps_0,A_0>0$.  Uniformly in
$X^{\eps_0}\leq K\leq X^{A_0}$,
\begin{align}
 &\frac{\mathfrak c_\varphi}{12}
 \sum_p(\log p)
 \sum_{\substack{k\geq2\\k\ {\rm even}}}(k-1)V(k/K)
 \int_0^\infty tW(t/X)
 F\!\left(\frac{16\pi^2p}{t(k-1)^2}\right)\M_k(p/t)\,dt \notag\\
 &\quad=\frac{\mathfrak c_\varphi\alpha}{768\pi^3}
 X^3I_V(K)\left(\int_0^\infty u^2W(u)\,du\right) \notag\\
 &\qquad\times
 \sum_{d,s\geq1}Q(d)\frac{d^3}{s^4}F((d/s)^2)
 +O_A(X^3K^4(\log X)^{-A}).                  \label{eq:bessel-to-atoms}
\end{align}
\end{proposition}

\begin{proof}
Insert \eqref{eq:Mk-intro}, put $n=k-1$, and, for fixed $p,t,d,s$, set
\[
 x=\frac{4\pi s\sqrt{p/t}}d,
 \qquad
 w(n)=V((n+1)/K)F\!\left(\frac{16\pi^2p}{tn^2}\right).
\]
Landau's uniform bounds \cite[(10), (12)]{LandauBesselBounds},
\[
                    |J_n(x)|\ll n^{-1/3},\qquad
                    |J_n(x)|\ll x^{-1/3},
\]
give the required absolute convergence.  For \(s\leq d\), the first
bound gives
\[
 \sum_{s\leq d}\frac1s
 \left|J_n\!\left(\frac{4\pi s\sqrt{p/t}}d\right)\right|
 \ll1+\log d.
\]
For \(s>d\), the second bound and \(\sqrt{p/t}\asymp K\) give
\[
 \sum_{s>d}\frac1s
 \left|J_n\!\left(\frac{4\pi s\sqrt{p/t}}d\right)\right|
 \ll K^{-1/3}d^{1/3}\sum_{s>d}s^{-4/3}\ll1.
\]
Consequently,
\[
 \sum_{d,s\geq1}\frac{Q(d)}s
 \left|J_n\!\left(\frac{4\pi s\sqrt{p/t}}d\right)\right|
 \ll \sum_{d\geq1}Q(d)(1+\log d)<\infty.
\]
Thus the Bessel series may be interchanged with the finite weight sum
and the smooth prime and level sums.  The inner sum over $k$ is
$T_w(x)$.  At $n=x$,
\begin{equation}\label{eq:atom-evaluation}
 w(x)=V((x+1)/K)F((d/s)^2).
\end{equation}
Write $\Delta_w(x)=T_w(x)-xw(x)/2$ and $y=p/t\asymp K^2$.
Proposition~\ref{prop:neumann} gives
\begin{equation}\label{eq:aggregated-neumann-error}
 \sqrt y\sum_{d,s\geq1}\frac{Q(d)}s
 \left|\Delta_w\!\left(\frac{4\pi s\sqrt y}{d}\right)\right|
 \ll1.
\end{equation}
Indeed, write $x=4\pi(s/d)\sqrt y$, so that $x\asymp K(s/d)$.
When $x\asymp K$, equivalently $s/d\asymp1$, the first error in
\eqref{eq:neumann-smoothed} is $O(K^{-1})$, and
\[
 \sqrt y\sum_dQ(d)\sum_{s\asymp d}\frac1{sK}
 \ll\sum_dQ(d).
\]
In the remaining range $x\leq K^{5/2}$, the rapid term contributes at most
\[
 K^{1-A}\sum_dQ(d)
 \sum_{s\ll dK^{3/2}}\frac1s
 \ll K^{1-A}\sum_dQ(d)(1+\log d+\log K),
\]
which is $O(1)$ after increasing $A$.  Finally, when $x>K^{5/2}$, so
$s/d\gg K^{3/2}$, the main term $xw(x)/2$ vanishes and
\eqref{eq:neumann-large-x} gives
\[
 K^2\sum_dQ(d)\sum_{s\gg dK^{3/2}}
 \frac1s\left(\frac ds\right)^A
 \ll_A K^{2-3A/2}\sum_dQ(d).
\]
This proves \eqref{eq:aggregated-neumann-error}.
After summing $p$ and integrating $t\asymp X$,
\eqref{eq:aggregated-neumann-error} gives
$O(X^3K^2)$, which is $K^{-2}$ relative to the normalizer.
Since \(K\geq X^{\eps_0}\), this term is absorbed by the error in
\eqref{eq:bessel-to-atoms}, after increasing its logarithmic exponent.

The resulting prime sum is
\[
 \sum_p(\log p)pV\!\left(\frac{4\pi s\sqrt{p/t}/d+1}{K}\right).
\]
Set
\[
 Y_{t,d,s}=\frac{tK^2d^2}{16\pi^2s^2},
 \qquad \psi_K(u)=uV(\sqrt u+K^{-1}).
\]
The preceding sum is
\[
             Y_{t,d,s}\sum_p(\log p)\psi_K(p/Y_{t,d,s}).
\]
For all sufficiently large \(K\), the functions \(\psi_K\) have a
common compact support in \((0,\infty)\) and uniformly bounded
\(C^1\)-seminorms.  Moreover, the condition \(F((d/s)^2)\ne0\) confines
\(d/s\) to a fixed compact subinterval of \((0,\infty)\).  Since
\(t\asymp X\),
we have \(Y_{t,d,s}\asymp XK^2\).  Lemma~\ref{lem:smooth-PNT} therefore
replaces the prime sum by
\[
 \int_0^\infty u
 V\!\left(\frac{4\pi s\sqrt{u/t}/d+1}{K}\right)du
\]
with the uniform error
\[
 O_A\!\left(t^2K^4(d/s)^4(\log X)^{-A}\right).
\]
Here \(\log(XK^2)\asymp_{A_0}\log X\), which is the only use of the
fixed upper exponent \(A_0\).
Moreover,
\begin{equation}\label{eq:pnt-label-sum}
 \sum_{d,s\geq1}Q(d)\frac{d^3}{s^4}
 \left|F((d/s)^2)\right|
 \ll_F\sum_{d\geq1}Q(d)<\infty,
\end{equation}
because \(s\asymp_F d\) and
\(\sum_{s\asymp d}d^3s^{-4}\ll1\).
The change of variable
$4\pi s\sqrt{u/t}/d=Kv-1$ gives
\begin{equation}\label{eq:prime-integral}
 \int_0^\infty uV\!\left(\frac{4\pi s\sqrt{u/t}/d+1}{K}\right)du
 =\frac{t^2d^4}{128\pi^4s^4}I_V(K).
\end{equation}
Since $\sqrt{p/t}\,x/2=2\pi sp/(td)$, substitution of
\eqref{eq:prime-integral} gives the main term in
\eqref{eq:bessel-to-atoms}.  The series remaining after
\eqref{eq:atom-evaluation} is absolutely convergent on the support of $F$;
this also justifies all interchanges.
\end{proof}

\section{From the Bessel series to the atomic measure}\label{sec:assembly}

\begin{lemma}[Grouping the atomic masses]\label{lem:grouping}
For $F\in C_c^\infty((0,\infty))$,
\begin{equation}\label{eq:grouping}
 \frac{\alpha}{2\pi}\sum_{d,s\geq1}Q(d)\frac{d^3}{s^4}F((d/s)^2)
 =\int_0^\infty F(v)\,d\mu_{\rm at}(v).
\end{equation}
\end{lemma}

\begin{proof}
Write $h=(d,s)$, $d=hq$, and $s=ha$.  Since $Q$ is supported on
squarefree integers, $h$ and $q$ are squarefree, $(h,q)=1$, and
$(q,a)=1$.  Hence the coefficient of $F((q/a)^2)$ on the left of
\eqref{eq:grouping} is
\begin{equation}\label{eq:grouped-coefficient}
 \frac{\alpha}{2\pi}\frac{q^3Q(q)}{a^4}
 \sum_{\substack{h\geq1\ {\rm squarefree}\\(h,q)=1}}\frac{Q(h)}h.
\end{equation}
Put
\[
 G_0=\prod_\ell\left(1+\frac{Q(\ell)}\ell\right).
\]
For every prime $\ell$,
\begin{equation}\label{eq:local-grouping}
 1+\frac{Q(\ell)}\ell
 =\frac{(\ell^2-1)^2}{D_\ell}.
\end{equation}
Thus \eqref{eq:grouped-coefficient} equals
\[
 \frac{\alpha G_0}{2\pi}\frac1{a^4}
 \prod_{\ell\mid q}\frac{\ell^5}{(\ell^2-1)^2}.
\]
Equations \eqref{eq:alpha-intro} and \eqref{eq:local-grouping} give
\[
 \frac{\alpha G_0}{2\pi}=W_0,
\]
which is the mass in \eqref{eq:atomic-measure-intro}.
\end{proof}

\begin{proof}[Proof of Theorem~\ref{thm:main}]
Choose \(\eta,\vartheta>0\) as in
Proposition~\ref{prop:prime-first-replacement}.  From
\eqref{eq:normalized-kernel-weight} and \eqref{eq:Phi-pr},
\begin{equation}\label{eq:Psi-Phi-bridge}
             P\Psi_{X,K,N,r}(p/P)=\Phi_{N,r}(p).
\end{equation}
The class-number formula, Lemma~\ref{lem:outer-tails}, and the removal
of the proper prime powers in the proof of
Proposition~\ref{prop:prime-first-replacement} therefore give
\begin{equation}\label{eq:numerator-to-prime-core}
 \mathcal N_{X,K}(F)
 =\frac P\pi\mathscr P_{X,K}^{\mathrm{tr}}
 +O\!\left(XP^2X^{-\delta}
           +XP^{3/2}(\log P)^C\right)
\end{equation}
for some \(\delta=\delta(\eps_0)>0\).
Proposition~\ref{prop:prime-first-replacement} now yields
\begin{align}
 \frac P\pi\left(
 \mathscr P_{X,K}^{\mathrm{tr}}
 -\mathscr Q_{X,K}^{\mathrm{tr}}\right)
 \ll{}&XP^2\bigg\{
 D\left(\frac TP\right)^{1/2}(\log P)^C
 +(\log P)^{-B}\bigg\}.                 \label{eq:core-replacement-scaled}
\end{align}
The first term in braces and the prime-power term in
\eqref{eq:numerator-to-prime-core} are negative powers of \(X\).
Proposition~\ref{prop:prime-main-identification} consequently gives
\[
 \mathcal N_{X,K}(F)
 =\sum_p(\log p)\mathcal M_{X,K}(p;F)
 +O_B(XP^2(\log X)^{-B}).
\]
Apply Proposition~\ref{prop:bessel-to-atoms}, divide by the asymptotic
in Proposition~\ref{prop:normalizer}, and use
Lemma~\ref{lem:grouping}.  This proves \eqref{eq:main-theorem}.

For \eqref{eq:denominator-theorem}, put
\[
 Y_{N,k}=N\left(\frac{k-1}{4\pi}\right)^2\asymp XK^2.
\]
Lemma~\ref{lem:smooth-PNT}, with the fixed test function \(F\), applies
uniformly in \(N,k\) to
\[
 \sum_p(\log p)F\!\left(\frac p{N((k-1)/(4\pi))^2}\right).
\]
Its main term is
\(Y_{N,k}\int_0^\infty F(v)\,dv\), and its error is
\(O_B(Y_{N,k}(\log X)^{-B})\).  Summation over \(N,k\) gives
\eqref{eq:denominator-theorem}, with total error
\(O_B(\mathcal C_{X,K}(\log X)^{-B})\).  The excluded primes dividing
\(N\) do not occur for large \(X\) by \eqref{eq:basic-scales}.
\end{proof}

\section*{Further questions}
The cubic endpoint in Theorem~\ref{thm:pointwise-main} comes from taking
absolute values after the exact-conductor decomposition.  It would be
interesting to determine whether cancellation between distinct nonzero
frequencies extends the pointwise asymptotic beyond $K=X^3$, or whether a
secondary transition occurs there.  The prime-averaged theorem suggests that
there is no corresponding obstruction after averaging in the prime.

Other natural problems are to remove part of the smoothing, to replace the
squarefree-level family by more general level structures, and to compare the
atomic law under natural and harmonic spectral weights in a common joint
level--weight regime.

\section*{Use of generative artificial intelligence}
The results and underlying arguments were developed by the author before the
use of generative artificial intelligence.  ChatGPT (OpenAI) was subsequently
used to assist with editorial revision, organization, and the expansion of
several proof details.  The author takes full responsibility for all
mathematical statements and citations.

\end{document}